\documentclass[12pt,twoside,american,english,british]{article}
\usepackage[T1]{fontenc}
\usepackage[utf8]{inputenc}
\usepackage[a4paper]{geometry}
\usepackage{color}
\usepackage{enumitem}
\usepackage{dsfont}
\usepackage{amsmath}
\usepackage{amsthm}
\usepackage{amssymb}
\usepackage{setspace}
\usepackage{esint}

\makeatletter
\numberwithin{equation}{section}
\numberwithin{figure}{section}
\theoremstyle{plain}
\newtheorem{thm}{\protect\theoremname}[section]
\theoremstyle{definition}
\newtheorem{defn}[thm]{\protect\definitionname}
\theoremstyle{plain}
\newtheorem{lem}[thm]{\protect\lemmaname}
\theoremstyle{plain}
\newtheorem{prop}[thm]{\protect\propositionname}
\theoremstyle{plain}
\newtheorem{cor}[thm]{\protect\corollaryname}
\newcommand{\lyxaddress}[1]{
	\par {\raggedright #1
	\vspace{1.4em}
	\noindent\par}
}

\usepackage[T1]{fontenc}
\usepackage[utf8]{inputenc}
\usepackage[a4paper]{geometry}
\usepackage{color}
\usepackage{dsfont}
\usepackage{amsmath}
\usepackage{amsthm}
\usepackage{amssymb}
\usepackage{setspace}
\usepackage{esint}

\makeatletter
\numberwithin{equation}{section}
\numberwithin{figure}{section}
\theoremstyle{plain}

\@ifundefined{date}{}{\date{}}
\usepackage{babel}
\usepackage{latexsym}
\usepackage{amsthm}
\usepackage{esint}
\usepackage{eucal}
\usepackage{epstopdf}
\usepackage{graphicx}
\usepackage{bigints}
\theoremstyle{plain}

\newtheoremstyle{boldremark}
    {\dimexpr\topsep/2\relax} % space above
    {\dimexpr\topsep/2\relax} % space below
    {}          % body font
    {}          % indent amount
    {\bfseries} % theorem head font
    {.}         % punctuation after theorem head
    {.5em}      % space after theorem head
    {}          % theorem hed spec. (empty = "normal")

\theoremstyle{boldremark}
\newtheorem{brem} [thm] {Remark} % remarks are numbered within sections

\usepackage{fancyhdr}
\usepackage{blindtext}
\usepackage{titlesec}

\titleformat{name=\chapter}[display]
{\normalfont\Huge\itshape}
{\titlerule[1pt]\vspace{-33pt}\filleft%
  \parbox[t]{6em}{%
    \raggedleft%
    \rule{\linewidth}{0.5ex}\newline%
    \chaptertitlename\ \thechapter%
  }%
}
{4pc}
{\normalfont\upshape\bfseries\Huge}
\allowdisplaybreaks

\makeatother

\usepackage{babel}
\addto\captionsbritish{\renewcommand{\definitionname}{Definition}}
\addto\captionsbritish{\renewcommand{\lemmaname}{Lemma}}
\addto\captionsbritish{\renewcommand{\theoremname}{Theorem}}
\addto\captionsenglish{\renewcommand{\definitionname}{Definition}}
\addto\captionsenglish{\renewcommand{\lemmaname}{Lemma}}
\addto\captionsenglish{\renewcommand{\theoremname}{Theorem}}
\providecommand{\definitionname}{Definition}
\providecommand{\lemmaname}{Lemma}
\providecommand{\theoremname}{Theorem}

\usepackage[colorlinks,pdfpagelabels,pdfstartview = FitH,bookmarksopen
= true,bookmarksnumbered = true,linkcolor = blue,plainpages =
false,hypertexnames = false,citecolor = blue,pagebackref=false,urlcolor=blue]{hyperref}

\usepackage{amsmath}  % for \iint macro
\usepackage{graphicx} % for \rotatebox macro
\def\Yint#1{\mathchoice
    {\YYint\displaystyle\textstyle{#1}}%
    {\YYint\textstyle\scriptstyle{#1}}%
    {\YYint\scriptstyle\scriptscriptstyle{#1}}%
    {\YYint\scriptscriptstyle\scriptscriptstyle{#1}}%
      \!\iint}
\def\YYint#1#2#3{{\setbox0=\hbox{$#1{#2#3}{\iint}$}
    \vcenter{\hbox{$#2#3$}}\kern-.51\wd0}}
\def\longdash{{-}\mkern-3.5mu{-}} 
\def\tiltlongdash{\rotatebox[origin=c]{15}{$\longdash$}}

\def\tiltfiint{\Yint\tiltlongdash}

\usepackage{etoolbox} % va messo nel preambolo
\makeatletter
\pretocmd{\@makefntext}{\setlength\parindent{0pt}\noindent}{}{%
  \PackageWarning{MyDoc}{Patch of \string\@makefntext\ failed}%
}
\makeatother

\makeatother

\usepackage{babel}
\addto\captionsamerican{\renewcommand{\corollaryname}{Corollary}}
\addto\captionsamerican{\renewcommand{\definitionname}{Definition}}
\addto\captionsamerican{\renewcommand{\lemmaname}{Lemma}}
\addto\captionsamerican{\renewcommand{\propositionname}{Proposition}}
\addto\captionsamerican{\renewcommand{\theoremname}{Theorem}}
\addto\captionsbritish{\renewcommand{\corollaryname}{Corollary}}
\addto\captionsbritish{\renewcommand{\definitionname}{Definition}}
\addto\captionsbritish{\renewcommand{\lemmaname}{Lemma}}
\addto\captionsbritish{\renewcommand{\propositionname}{Proposition}}
\addto\captionsbritish{\renewcommand{\theoremname}{Theorem}}
\addto\captionsenglish{\renewcommand{\corollaryname}{Corollary}}
\addto\captionsenglish{\renewcommand{\definitionname}{Definition}}
\addto\captionsenglish{\renewcommand{\lemmaname}{Lemma}}
\addto\captionsenglish{\renewcommand{\propositionname}{Proposition}}
\addto\captionsenglish{\renewcommand{\theoremname}{Theorem}}
\providecommand{\corollaryname}{Corollary}
\providecommand{\definitionname}{Definition}
\providecommand{\lemmaname}{Lemma}
\providecommand{\propositionname}{Proposition}
\providecommand{\theoremname}{Theorem}

\begin{document}
\title{\textbf{Widely degenerate anisotropic diffusion:\\local boundedness
and semicontinuity}}
\author{Pasquale Ambrosio, Simone Ciani, Giovanni Cupini}
\date{\noindent October 20, 2026 }
\maketitle
\begin{abstract}
\begin{singlespace}
\noindent We investigate the regularity of local weak solutions to
evolution equations of the form 
\[
\partial_{t}u\,=\,\sum_{i=1}^{n}\,\partial_{x_{i}}\left[a_{i}(x,t)\,(\vert\partial_{x_{i}}u\vert-\delta_{i})_{+}^{p_{i}-1}\,\frac{\partial_{x_{i}}u}{\vert\partial_{x_{i}}u\vert}\right]\,\,\,\,\,\,\,\,\,\,\mathrm{in}\,\,\,\Omega_{T}\,=\,\Omega\times(0,T)\,,
\]
where $\Omega$ is a bounded domain in $\mathbb{R}^{n}$ with $n\geq2$,
the coefficients $a_{i}$ are measurable and bounded, $p_{i}>1$ and
$\delta_{i}\geq0$ are fixed parameters. Under suitable assumptions
on the exponents $p_{i}$, we first show that the local boundedness
of weak solutions follows from their membership in an appropriate
non-homogeneous parabolic De Giorgi class. We then establish the existence
of semicontinuous representatives for local weak sub(super)-solutions.
Our analysis extends analogous results available for less degenerate
operators and generalizes the local boundedness results obtained in
\cite{AmbCiani} to fully anisotropic, widely degenerate parabolic
PDEs with non-smooth coefficients depending additionally on the space-time
variables $(x,t)$, whose growth is governed by a family of exponents
$p_{i}$ rather than by a single exponent.\vspace{0.2cm}
\end{singlespace}
\end{abstract}

\selectlanguage{english}%
\begin{singlespace}

\section{Introduction}
\end{singlespace}

\selectlanguage{british}%
\begin{singlespace}
\noindent $\hspace*{1em}$Let $\Omega$ be a bounded domain in $\mathbb{R}^{n}$,
$n\geq2$, and let $T\in(0,\infty)$. We are concerned with local
regularity properties of weak sub(super)-solutions to the following
parabolic equation:
\begin{equation}
\partial_{t}u\,=\,\sum_{i=1}^{n}\,\partial_{x_{i}}\left[a_{i}(x,t)\,(\vert\partial_{x_{i}}u\vert-\delta_{i})_{+}^{p_{i}-1}\,\frac{\partial_{x_{i}}u}{\vert\partial_{x_{i}}u\vert}\right]\,\,\,\,\,\,\,\,\,\,\mathrm{in}\,\,\,\Omega_{T}=\Omega\times(0,T)\,,\label{eq:equation}
\end{equation}
where $p_{1},\ldots,p_{n}>1$, $\delta_{1},\ldots,\delta_{n}$ are
non-negative real numbers, and $\left(\,\cdot\,\right)_{+}$ stands
for the positive part. Here $a_{i}$,\textit{ $i\in\{1,\ldots,n\}$},
are measurable coefficients satisfying
\begin{equation}
\Lambda^{-1}\,\leq\,a_{i}(x,t)\,\leq\,\Lambda\,\,\,\,\,\,\,\,\,\,\mathrm{a.e.}\,\,\,\mathrm{in}\,\,\,\Omega_{T}\,,\label{eq:coeff}
\end{equation}
for some constant $\Lambda\geq1$.\\
$\hspace*{1em}$Evolution equations of a form similar to (\ref{eq:equation})
have been studied since the 1960s, especially by the Soviet school;
see, for instance, the paper \cite{Vishik} by Vishik. The above equation,
in the special case $p_{i}=p$, $\delta_{i}=0$ and $a_{i}=1$ for
all $i\in\{1,\ldots,n\}$, also appears explicitly in the monographs
\cite{Lions}, \cite[Example 4.A, Chap. III]{Show} and \cite[Example 30.8]{Zeid},
among others. More recently, equations of the type (\ref{eq:equation}),
with $p_{i}=p\geq2$ and $a_{i}=1$ for all $i\in\{1,\ldots,n\}$,
have been considered in the works \cite{AmbCiani} and \cite{AmbAnis}:
in the former, we have established the local boundedness of local
weak solutions, while in the latter we have proved that such solutions
are locally Lipschitz continuous in the spatial variable (see also
\cite{BBLVpar} for the case in which all $\delta_{i}$ are equal
to zero).\\
The equation considered in this paper combines, in a unified way,
two of the most studied classes of nonlinear equations in recent years:
the \textit{widely degenerate equations} and the \textit{anisotropic}
ones.\\
$\hspace*{1em}$Widely degenerate equations have attracted growing
attention in the literature; see, for instance, \cite{Amb00,Amb2,AmGrPa,BDGPell,Brasco,BraCarSan,CGHP,CoFi,CGGP,Grim,GriRus,Mons,Picc,Russo,Strunk-ell}
in the elliptic setting and \cite{Amb1,Amb3,AmbBau,AmbCuDe,AmbPass,BoDuGiPa,GenPas,Strunk}
in the parabolic context. In these latter works, the model equation
under consideration is the following:
\begin{equation}
\partial_{t}u-\mathrm{div}\left((\vert Du\vert-\lambda)_{+}^{p-1}\frac{Du}{\vert Du\vert}\right)=f\,,\,\,\,\,\,\,\,\,\mathrm{with}\,\,\lambda>0\,.\label{eq:AmbPass}
\end{equation}
According to the now-standard terminology, this equation is widely
degenerate, in the sense that the diffusion part is uniformly elliptic
only outside a ball with radius $\lambda$, while it behaves asymptotically,
that is, for large values of $\vert Du\vert$, like the parabolic
$p$-Laplace operator. Therefore, equations of the form (\ref{eq:AmbPass})
fall within the class of \textit{asymptotically regular parabolic
problems} (for a comprehensive overview of this topic, see \cite{Amb1,AmbPass,BoDuGiPa}
and the references therein). As already pointed out in \cite{Amb3,BoDuGiPa},
no more than Lipschitz regularity can be expected for solutions to
equation (\ref{eq:AmbPass}). In fact, when $f=0$, any time-independent
$\lambda$-Lipschitz function solves (\ref{eq:AmbPass}), and even
more, it is a solution of the associated stationary equation.\\
Returning to our setting, equation (\ref{eq:equation}) exhibits a
similar type of degeneracy as (\ref{eq:AmbPass}), since the ellipticity
of the corresponding operator in divergence form breaks down on the
set 
\[
\bigcup_{i=1}^{n}\,\{\vert\partial_{x_{i}}u\vert\leq\delta_{i}\}\,.
\]
In particular, equation (\ref{eq:equation}) displays a threshold-type
behavior (possibly with different thresholds $\delta_{i}$), with
regions where the diffusion coefficients may vanish identically, thereby
resulting in a loss of the parabolic structure on sets of possibly
positive measure.\\
$\hspace*{1em}$Anisotropic equations and functionals were introduced
in the 1980s by Giaquinta \cite{Giaq} and Marcellini \cite{Marc}
and have since been extensively studied under many different aspects;
see, e.g., \cite{AmbCupMas,BiaCupMas2,BB1,BouBra,BouBraLeo,Cianchi,CiHeSk,CSV,CuMaMa1,CuMaMa2,CuMaMa3,CuMaMa4,FeoPasPos,FusSbo1,FusSbo2,GriRuss2,Russo2,Strof}
in the elliptic setting and \cite{BDM,BDMS,BoDuMin,CiaGuaVes,CHSS,CiaHenSkr,CiaMosVes,DeFil,TerTer}
in the parabolic framework. The term \textit{anisotropic} refers to
the fact that the diffusion is of power type with exponents that may
vary with the spatial directions. While a fairly well-developed regularity
theory is available for equations with differentiable coefficients
(see, for instance, \cite{BouBra} and \cite{EleMarMas}), the situation
is much less understood in the case of rough coefficients. In this
latter setting, although the $L^{\infty}$-estimates and the Critical
Mass Lemma, two key tools necessary to prove regularity (see (\ref{eq:sup2}),
(\ref{eq:estcor}) and Lemma \ref{lem:critmass} later on), can still
be adapted, the same does not apply for the so-called shrinking lemma
(see \cite[Lemma 7.2 Chap. III and Lemma 5.1 Chap. IV]{DiBe}). We
recall that the shrinking lemma, originally introduced by De Giorgi
\cite{DeGio}, is another fundamental ingredient needed to demonstrate
regularity. Roughly speaking, it ensures that, upon restriction to
smaller domains, if we denote by $\mu$ the infimum of the solution,
the measure of the set where the solution takes values between $\mu$
and $\mu+\varepsilon$ tends to zero as $\varepsilon\to0$ (see \cite{DBUV}
for further details). A partial breakthrough in this direction was
obtained by Liskevich and Skrypnik \cite{LisSkr}, who succeeded in
proving regularity in a very special case by replacing the shrinking
lemma with the positivity expansion method introduced in \cite{DBeGiVe}
in the parabolic framework. However, despite these recent advances,
the regularity theory remains highly fragmented. For example, Harnack's
estimates for the elliptic operators are currently available, to our
knowledge, only for operators with constant coefficients and with
exponents $p_{i}>2$ satisfying a suitable parabolic condition (see
\cite{CiaMosVes} for more details).\\
$\hspace*{1em}$As for widely degenerate anisotropic problems, while
in the elliptic setting they have already been studied from several
viewpoints (see, for instance, \cite{BouBraJul,BBLVsup,BraCar,BrCa2,BLPV}
and the references therein), in the parabolic context only a few regularity
results are currently available (we refer again to \cite{AmbAnis,AmbCiani}),
and many aspects still remain to be investigated.\\
$\hspace*{1em}$Before describing the aims and structure of this paper,
we also mention the work \cite{Ciani}, where the authors establish
several qualitative properties of weak solutions to doubly nonlinear
anisotropic evolution equations, whose prototype is 
\begin{equation}
\partial_{t}(\vert u\vert^{\alpha-1}\,u)-\sum_{i=1}^{n}\,\partial_{x_{i}}(\vert\partial_{x_{i}}u\vert^{p_{i}-2}\,\partial_{x_{i}}u)\,=\,0\,,\,\,\,\,\,\,\,\,\mathrm{with}\,\,\alpha>0\,.\label{eq:doublyNL}
\end{equation}
Among other things, they show that certain regularity properties,
such as local boundedness and the existence of semicontinuous representatives
for weak solutions, are actually embodied in suitable energy estimates
rather than in the mere class of solutions to the equation. However,
their framework does not cover our widely degenerate operator. This
is because, from the energetic viewpoint, equation (\ref{eq:equation})
exhibits features that are not captured by (\ref{eq:doublyNL}). Roughly
speaking, equation (\ref{eq:equation}) behaves as if a lower-order
term on the right-hand side of (\ref{eq:doublyNL}) were capable of
suppressing the diffusion when the relevant energy is below a certain
threshold; see, for instance, the bound for the energies $\mathcal{E}_{j}$
in the proof of Lemma \ref{lem:critmass} below.\\
$\hspace*{1em}$In this paper, we not only extend the aforementioned
regularity properties to local weak sub(super)-solutions of (\ref{eq:equation}),
but we also generalize the local boundedness results obtained in \cite{AmbCiani}
to fully anisotropic and widely degenerate operators, with non-smooth
coefficients depending also on the space-time variables $(x,t)$,
and whose growth is prescribed by a whole family of exponents $p_{i}>1$,
$i\in\{1,\ldots,n\}$, rather than by a single exponent $p\geq2$.\\
\\
\smallskip{}

\end{singlespace}
\begin{singlespace}

\subsection{Plan of the paper}
\end{singlespace}

\begin{singlespace}
\noindent $\hspace*{1em}$Concerning local weak sub(super)-solutions
to (\ref{eq:equation}), we carry out a detailed analysis of
\end{singlespace}
\begin{enumerate}
\begin{singlespace}
\item the weak formulation of the notion of sub(super)-solution via Steklov
averages;\vspace{-2mm}
\item energy estimates for sub(super)-solutions;\vspace{-2mm}
\item local boundedness and the existence of semicontinuous representatives
(for the whole energy classes defined at the end of Section \ref{sec:energy_estimate}).
\end{singlespace}
\end{enumerate}
\begin{singlespace}
\noindent In what follows, we introduce each one of these aspects.
After presenting the definition of weak sub(super)-solution, in Subsection
\ref{subsec:SteklovAverages} we give a weak formulation of sub(super)-solution
based on the usual Steklov averaging technique. This is then employed
to show that local weak sub(super)-solutions to (\ref{eq:equation})
are elements of special energy classes.
\end{singlespace}
\begin{singlespace}

\subsection*{Energy classes}
\end{singlespace}

\begin{singlespace}
\noindent $\hspace*{1em}$In this paper, we follow an approach dating
back to De Giorgi (see \cite{DeGio}) and subsequently widely developed
in the literature: we show that certain regularity properties are,
in general, encoded in suitable energy estimates rather than in the
mere class of local weak sub(super)-solutions to (\ref{eq:equation}).
More precisely, here we define two classes of functions $\mathcal{DG}^{+}(\mathbf{p},\{\delta_{i}\},\Omega_{T},\mathcal{C})$
and $\mathcal{DG}^{-}(\mathbf{p},\{\delta_{i}\},\Omega_{T},\mathcal{C})$
(see the end of Section \ref{sec:energy_estimate}), which we believe
to be of fundamental importance in the study of the local behavior
of local weak sub(super)-solutions. The inclusion of the latter in
the two aforementioned classes, explicitly observed in Remark \ref{def:ossDG}
below, allows us to state results of broader scope and application
within the regularity theory for widely degenerate and anisotropic
operators. Indeed, the results obtained here can be formulated directly
for elements of the (larger) classes $\mathcal{DG}^{\pm}(\mathbf{p},\{\delta_{i}\},\Omega_{T},\mathcal{C})$,
rather than merely in terms of local weak sub(super)-solutions to
(\ref{eq:equation}).\\
\smallskip{}

\end{singlespace}
\begin{singlespace}

\subsection*{Local boundedness and semicontinuity}
\end{singlespace}

\begin{singlespace}
\noindent $\hspace*{1em}$Regarding the boundedness of local weak
sub(super)-solutions to (\ref{eq:equation}), we are interested in
the parameter range 
\begin{equation}
1\,<\,p_{i}\,<\,\overline{p}\left(1+\frac{2}{n}\right),\,\,\,\,\,\,\,\,\,\,\,\,\,\overline{p}\,<\,n\,,\label{eq:range}
\end{equation}
where\foreignlanguage{english}{
\[
\overline{p}\,:=\left(\frac{1}{n}\,\sum_{i=1}^{n}\frac{1}{p_{i}}\right)^{-1}.
\]
}

\noindent More generally, we prove that, under certain conditions
involving the exponents $p_{i}$, the functions in the energy class
$\mathcal{DG}^{+}(\mathbf{p},\{\delta_{i}\},\Omega_{T},\mathcal{C})$
(resp. $\mathcal{DG}^{-}(\mathbf{p},\{\delta_{i}\},\Omega_{T},\mathcal{C})$)
are locally essentially bounded from above (resp. below) in $\Omega_{T}$.
Hence, by the inclusion of local weak solutions of (\ref{eq:equation})
in the classes $\mathcal{DG}^{\pm}(\mathbf{p},\{\delta_{i}\},\Omega_{T},\mathcal{C})$,
we obtain the local boundedness for weak solutions.\\
$\hspace*{1em}$In our proofs, we need to distinguish between the
cases $\overline{p}>\frac{2n}{n+2}$ and $\overline{p}\leq\frac{2n}{n+2}$.
In the latter case, we require the extra integrability condition 
\begin{equation}
u\,\in\,L_{loc}^{m}(\Omega_{T})\,,\,\,\,\,\mathrm{for\,\,some\,\,}m\,>\,\frac{n}{\overline{p}}\,(2-\overline{p})\,,\label{eq:extra_integr}
\end{equation}
similarly to subcritical $p$-Laplacian equations (see, for instance,
\cite{DiBe}). We note that the two ranges for $\overline{p}$ are
identical to those appearing in \cite{Yu-Lian}, and also the extra
integrability condition (\ref{eq:extra_integr}) coincides with the
integrability assumption used in \cite{Yu-Lian}.\\
$\hspace*{1em}$Furthermore, we show that local weak sub(super)-solutions
to (\ref{eq:equation}) admit an upper (lower) semicontinuous representative
in $\Omega_{T}$. In fact, the same result holds more generally for
elements of the classes $\mathcal{DG}^{+}(\mathbf{p},\{\delta_{i}\},\Omega_{T},\mathcal{C})$
and $\mathcal{DG}^{-}(\mathbf{p},\{\delta_{i}\},\Omega_{T},\mathcal{C})$,
respectively, under appropriate assumptions on the exponents $p_{i}$,
and additionally assuming that condition (\ref{eq:extra_integr})
holds in the case $\overline{p}\leq2n/(n+2)$. To prove this, we adapt
to our anisotropic setting the very general method developed in \cite{Liao}.
There, it is shown that the existence of lower or upper semicontinuous
representatives is a consequence of some kind of measure-theoretical
maximum principle, referred to in the literature as a Critical Mass
Lemma, or a De Giorgi-type Lemma. In establishing this result, we
follow the approach of \cite{CiaGuaVes}; however, the novelty here
lies in the consideration of an alternative formulation embodying
the non-homogeneity of the energy classes $\mathcal{DG}^{\pm}$. Therefore,
differently from \cite{CiaGuaVes}, the Critical Mass Lemma proved
here holds under an additional condition involving the exponents $p_{i}$,
the degeneracy thresholds $\delta_{i}$, and the parameters $M,\rho$
defining the intrinsic cylinders used in the proof (see Lemma \ref{lem:critmass}
below).\\
\smallskip{}

\end{singlespace}

\subsection{Structure of the paper}

$\hspace*{1em}$The paper is organized as follows. Section \ref{sec:prelim}
introduces the notion of local weak sub(super)-solution and the associated
function spaces, and collects the necessary preliminary material,
including standard notation, the main embedding results and Steklov
averages. In Section \ref{sec:energy_estimate} we derive the main
energy estimates and, accordingly, we define the functional classes
$\mathcal{DG}^{+}(\mathbf{p},\{\delta_{i}\},\Omega_{T},\mathcal{C})$
and $\mathcal{DG}^{-}(\mathbf{p},\{\delta_{i}\},\Omega_{T},\mathcal{C})$.
Next, in Section \ref{sec:boundedness}, we study the local boundedness
from above or below of functions belonging to the classes $\mathcal{DG}^{\pm}(\mathbf{p},\{\delta_{i}\},\Omega_{T},\mathcal{C})$.
Lastly, in Section \ref{sec:DeGiorgiLemma} we establish a Critical
Mass Lemma, which is then used to prove the existence of semicontinuous
representatives for local weak sub(super)-solutions to (\ref{eq:equation}).
The final \hyperref[sec:appendice]{Appendix} provides details on
the construction of suitable cut-off functions employed in our proofs.
\selectlanguage{english}%
\begin{singlespace}

\section{Notation and preliminaries\label{sec:prelim}}
\end{singlespace}

\selectlanguage{british}%
\noindent $\hspace*{1em}$In this paper we shall denote by $C$ or
$c$ a general positive constant that may vary on different occasions,
even within the same line of estimates. Relevant dependencies on parameters
and special constants will be suitably emphasized using parentheses
or subscripts. The norm we use on $\mathbb{R}^{k}$, $k\in\mathbb{N}$,
will be the standard Euclidean one and it will be denoted by $\left|\,\cdot\,\right|$.
In particular, for the vectors $\xi,\eta\in\mathbb{R}^{k}$, we write
$\langle\xi,\eta\rangle$ for the usual inner product and $\left|\xi\right|:=\langle\xi,\xi\rangle^{\frac{1}{2}}$
for the corresponding Euclidean norm.\\
$\hspace*{1em}$For $\rho>0$ and $x_{0}\in\mathbb{R}^{n}$, we define
the following $n$-dimensional hyper-rectangles:
\[
K_{\rho}:=\,\prod_{i=1}^{n}\,(-\rho^{\frac{1}{p_{i}}},\rho^{\frac{1}{p_{i}}})\,,
\]
\[
K_{\rho}(x_{0}):=\,\prod_{i=1}^{n}\,(x_{0,i}-\rho^{\frac{1}{p_{i}}},x_{0,i}+\rho^{\frac{1}{p_{i}}})\,.
\]
Moreover, for $\rho>0$ we define the space-time cylinder 
\[
Q_{\rho}:=\,K_{\rho}\times(-\rho,0)\,,
\]
and if $(x_{0},t_{0})\in\mathbb{R}^{n+1}$, we let $Q_{\rho}(x_{0},t_{0})$
denote the cylinder with vertex at $(x_{0},t_{0})$ congruent to $Q_{\rho}$,
i.e.,
\[
Q_{\rho}(x_{0},t_{0}):=\,K_{\rho}(x_{0})\times(t_{0}-\rho,t_{0})\,.
\]
For a general cylinder $Q=B\times(t_{0},t_{1})$, where $B\subset\mathbb{R}^{n}$
and $t_{0}<t_{1}$, we denote by 
\[
\partial_{\mathrm{par}}Q:=\,(\overline{B}\times\{t_{0}\})\cup(\partial B\times(t_{0},t_{1}))
\]
the usual \textit{parabolic boundary} of $Q$, which is nothing but
its standard topological boundary without the upper cap $\overline{B}\times\{t_{1}\}$.

\selectlanguage{english}%
\noindent $\hspace*{1em}$If $E\subseteq\mathbb{R}^{k}$ is a Lebesgue-measurable
set, then we will denote by $\vert E\vert$ its $k$-dimensional Lebesgue
measure. \foreignlanguage{american}{When $0<\vert E\vert<\infty$,
the mean value of a function $v\in L^{1}(E)$ is defined by 
\[
\fint_{E}v(y)\,dy\,:=\,\frac{1}{\vert E\vert}\int_{E}v(y)\,dy\,.
\]
}\foreignlanguage{british}{$\hspace*{1em}$Given a vector of real
numbers $\mathbf{p}=(p_{1},\ldots,p_{n})$ with $p_{i}>1$ for every
$i\in\{1,\ldots,n\}$, we define the following spaces}\\
\foreignlanguage{british}{
\[
W_{0}^{1,\mathbf{p}}(\Omega):=\left\{ v\in W_{0}^{1,1}(\Omega):\partial_{x_{i}}v\in L^{p_{i}}(\Omega),\,\,\mathrm{for\,\,all}\,\,i=1,\ldots,n\right\} ,
\]
\vspace{0.5mm}
\[
W_{loc}^{1,\mathbf{p}}(\Omega):=\left\{ v\in W_{loc}^{1,1}(\Omega):\partial_{x_{i}}v\in L_{loc}^{p_{i}}(\Omega),\,\,\mathrm{for\,\,all}\,\,i=1,\ldots,n\right\} ,
\]
\vspace{0.5mm}
\[
L^{\mathbf{p}}(0,T;W_{0}^{1,\mathbf{p}}(\Omega)):=\left\{ v\in L^{1}(0,T;W_{0}^{1,1}(\Omega)):\partial_{x_{i}}v\in L^{p_{i}}(\Omega_{T}),\,\,\mathrm{for\,\,all}\,\,i=1,\ldots,n\right\} ,
\]
\vspace{0.5mm}
\[
L_{loc}^{\mathbf{p}}(0,T;W_{0}^{1,\mathbf{p}}(\Omega)):=\left\{ v\in L_{loc}^{1}(0,T;W_{0}^{1,1}(\Omega)):\partial_{x_{i}}v\in L_{loc}^{p_{i}}(0,T;L_{loc}^{p_{i}}(\Omega)),\,\,\mathrm{for\,\,all}\,\,i=1,\ldots,n\right\} ,
\]
\vspace{0.5mm}
\[
L_{loc}^{\mathbf{p}}(0,T;W_{loc}^{1,\mathbf{p}}(\Omega)):=\left\{ v\in L_{loc}^{1}(0,T;W_{loc}^{1,1}(\Omega)):\partial_{x_{i}}v\in L_{loc}^{p_{i}}(\Omega_{T}),\,\,\mathrm{for\,\,all}\,\,i=1,\ldots,n\right\} .
\]
\vspace{0.5mm}
}

\noindent Throughout this paper, we let
\[
\mathcal{P}:=\,\max\,\{p_{1},\ldots,p_{n}\}\,\,\,\,\,\,\,\,\,\,\,\,\,\,\,\mathrm{and}\,\,\,\,\,\,\,\,\,\,\,\,\,\,\,P:=\,\max\,\{2,\mathcal{P}\}\,.
\]
$\hspace*{1em}$Now \foreignlanguage{british}{let $F:\Omega_{T}\times\mathbb{R}^{n}\to\mathbb{R}$
be the function defined by
\begin{equation}
F(x,t,\xi):=\sum_{i=1}^{n}\,\frac{a_{i}(x,t)}{p_{i}}\,(\vert\xi_{i}\vert-\delta_{i})_{+}^{p_{i}}\,.\label{eq:F}
\end{equation}
}In this work, we define a local weak sub(super)-solution of (\ref{eq:equation})
as follows.\medskip{}

\selectlanguage{british}%
\begin{defn}
\noindent \label{def:weaksoldef}A function ${\color{red}{\normalcolor u\in C_{loc}^{0}(0,T;L_{loc}^{2}(\Omega))\cap L_{loc}^{\mathbf{p}}(0,T;W_{loc}^{1,\mathbf{p}}(\Omega))\cap L_{loc}^{P}(\Omega_{T})}}$
is a \textit{local weak sub(super)-solution} of (\ref{eq:equation})
in $\Omega_{T}$ if
\begin{equation}
\iint_{\Omega_{T}}\left(u\,\partial_{t}\varphi\,-\langle D_{\xi}F(x,t,Du),D\varphi\rangle\right)dx\,dt\,\ge\,(\leq)\,0\label{eq:locweaksol}
\end{equation}
for all test functions $\varphi\in C_{0}^{\infty}(\Omega_{T})$, $\varphi\geq0$. 
\end{defn}

\begin{defn}
\noindent \label{def:locweaksol2}A function ${\color{red}{\normalcolor u\in C_{loc}^{0}(0,T;L_{loc}^{2}(\Omega))\cap L_{loc}^{\mathbf{p}}(0,T;W_{loc}^{1,\mathbf{p}}(\Omega))\cap L_{loc}^{P}(\Omega_{T})}}$
is a \textit{local weak solution} of (\ref{eq:equation}) in $\Omega_{T}$
if
\[
\iint_{\Omega_{T}}\left(u\,\partial_{t}\varphi\,-\langle D_{\xi}F(x,t,Du),D\varphi\rangle\right)dx\,dt\,=\,0
\]
for all test functions $\varphi\in C_{0}^{\infty}(\Omega_{T})$. 
\end{defn}

\noindent $\hspace*{1em}$Given a real number $\ell\geq1$, $\ell^{*}$
is its Sobolev exponent, i.e. \foreignlanguage{english}{
\begin{equation}
\ell^{\ast}:=\begin{cases}
\begin{array}{cc}
\frac{n\ell}{n-\ell} & \mathrm{if}\,\,\,\ell<n,\\
\mathrm{any\,\,value\,\,in}\,\,(\ell,\infty) & \mathrm{if}\,\,\,\ell\geq n,
\end{array}\end{cases}\label{eq:SobConj}
\end{equation}
and $\ell'$ is the conjugate exponent of $\ell$, i.e. $\ell'=\frac{\ell}{\ell-1}$
if $\ell>1$, while $\ell'=\infty$ if $\ell=1$. Moreover, we denote
by $\overline{p}$ the harmonic average of $\mathbf{p}=(p_{1},\ldots,p_{n})$,
i.e.
\[
\overline{p}\,:=\left(\frac{1}{n}\,\sum_{i=1}^{n}\frac{1}{p_{i}}\right)^{-1}.
\]
}

\noindent $\hspace*{1em}$Next, we recall some useful anisotropic
Sobolev embeddings.
\begin{lem}[Sobolev-Troisi embedding, \cite{DMV}]
\noindent Let $\mathcal{K}\subseteq\mathbb{R}^{n}$ be a rectangular
domain, $\overline{p}<n$ and $\alpha_{i}>0$, $i\in\{1,\ldots,n\}.$
If we define 
\[
\boldsymbol{\alpha}\,=\,(\alpha_{1},\ldots,\alpha_{n})\,,\,\,\,\,\,\,\,\,\,\,\tilde{\alpha}\,=\sum_{i=1}^{n}\alpha_{i}\,,\,\,\,\,\,\,\,\,\,\,p_{\alpha}^{*}\,=\,\overline{p}^{*}\,\frac{\tilde{\alpha}}{n}\,,
\]
then there exists a constant $C=C(n,\mathbf{p},\boldsymbol{\alpha})>0$
such that 
\[
\Vert v\Vert_{L^{p_{\alpha}^{*}}(\mathcal{K})}\,\leq\,C\,\prod_{i=1}^{n}\Vert\partial_{x_{i}}\vert v\vert^{\alpha_{i}}\Vert_{L^{p_{i}}(\mathcal{K})}^{\frac{1}{\tilde{\alpha}}}
\]
for all $v\in W_{0}^{1,\mathbf{p}}(\mathcal{K})$.
\end{lem}

\noindent As a corollary, one gets the usual Troisi inequality (see
\cite{Troisi}).

\noindent \begin{brem}\label{SobEmb}Let $\mathcal{K}\subseteq\mathbb{R}^{n}$
be a rectangular domain. Then there exists a constant $C=C(n,\mathbf{p})>0$
such that
\begin{equation}
\Vert v\Vert_{L^{\overline{p}^{*}}(\mathcal{K})}\,\leq\,C\,\prod_{i=1}^{n}\Vert\partial_{x_{i}}v\Vert_{L^{p_{i}}(\mathcal{K})}^{\frac{1}{n}}\label{eq:anisEmb}
\end{equation}
for all $v\in W_{0}^{1,\mathbf{p}}(\mathcal{K})$. Moreover, taking
both sides of (\ref{eq:anisEmb}) to the exponent $\overline{p}$
and using Young's inequality on the right-hand side with exponents
$\frac{n\,p_{i}}{\overline{p}}$, whose inverses add up to $1$ due
to the definition of $\overline{p}$, we obtain the following useful
inequality:
\begin{equation}
\left(\int_{\mathcal{K}}\vert v\vert^{\overline{p}^{*}}\,dx\right)^{\frac{\overline{p}}{\overline{p}^{*}}}\leq\,C\,\sum_{i=1}^{n}\int_{\mathcal{K}}\vert\partial_{x_{i}}v\vert^{p_{i}}\,dx\,,\,\,\,\,\,\,\,\,v\in W_{0}^{1,\mathbf{p}}(\mathcal{K})\,.\label{eq:Troisi}
\end{equation}
\end{brem}

\noindent $\hspace*{1em}$The next lemma plays a crucial role in the
De Giorgi-type iterations; we refer to \cite[Lemma 7.1]{Giu} for
a proof.
\begin{lem}
\label{lem:Giusti}Let $\mu>0$ and let $\{Y_{j}\}_{j\,\in\,\mathbb{N}_{0}}$
be a sequence of non-negative real numbers, satisfying the recursive
inequalities 
\[
Y_{j+1}\,\leq\,C\,b^{j}\,Y_{j}^{1+\mu}
\]
where $C>0$ and $b>1$. If $Y_{0}\leq C^{-\,\frac{1}{\mu}}\,b^{-\,\frac{1}{\mu^{2}}}$,
then 
\[
\lim_{j\to\infty}Y_{j}=0\,.
\]
\end{lem}

\selectlanguage{american}%

\subsection{Steklov averages\label{subsec:SteklovAverages}}

\selectlanguage{british}%
$\hspace*{1em}$\foreignlanguage{american}{In this section, we recall
the definition and some elementary properties of Steklov averages.
Let us denote a domain in space-time by $Q':=\Omega'\times I$, where
$\Omega'\subseteq\Omega$ is a bounded domain and $I:=(t_{0},t_{1})\subseteq(0,T)$.
For every $h\in(0,t_{1}-t_{0})$ and $v\in L^{1}(Q',\mathbb{R}^{k})$,
where $k\in\mathbb{N}$, the \textit{Steklov average} $[v]_{h}(\cdot,t)$
is defined by 
\[
[v]_{h}(x,t):=\begin{cases}
\begin{array}{cc}
{\displaystyle \frac{1}{h}\int_{t}^{t+h}v(x,s)\,ds} & \,\,\,\text{if }\,t\in(t_{0},t_{1}-h],\\
0 & \,\,\,\text{if }\,t\in(t_{1}-h,t_{1}),
\end{array}\end{cases}
\]
for $x\in\Omega'$. This definition implies, for almost every $(x,t)\in\Omega'\times(t_{0},t_{1}-h)$,
\[
\frac{\partial[v]_{h}}{\partial t}(x,t)\,=\,\frac{v(x,t+h)-v(x,t)}{h}\,.
\]
}

\noindent $\hspace*{1em}$\foreignlanguage{american}{The proof of
the following result is straightforward from the theory of Lebesgue
spaces (see \cite[Lemma 3.2, Chapter I]{DiBe}).\medskip{}
}
\selectlanguage{american}%
\begin{lem}
\label{lem:Stek}Let $q,r\geq1$ and $v\in L^{r}\left(t_{0},t_{1};L^{q}(\Omega')\right)$.
Then, as $h\rightarrow0$, $[v]_{h}$ converges to $v$ in $L^{r}\left(t_{0},t_{1}-\varepsilon;L^{q}(\Omega')\right)$
for every $\varepsilon\in(0,t_{1}-t_{0})$. If $v\in C^{0}\left(t_{0},t_{1};L^{q}(\Omega')\right)$,
then as $h\rightarrow0$, $[v]_{h}(\cdot,t)$ converges to $v(\cdot,t)$
in $L^{q}(\Omega')$ for every $t\in(t_{0},t_{1}-\varepsilon)$, $\forall\,\,\varepsilon\in(0,t_{1}-t_{0})$.
\end{lem}

\selectlanguage{british}%
\noindent $\hspace*{1em}$\foreignlanguage{american}{A very useful
formulation, equivalent to (\ref{eq:locweaksol}), is the one involving
Steklov averages. Assume that ${\color{red}{\normalcolor u\in C_{loc}^{0}(0,T;L_{loc}^{2}(\Omega))\cap L_{loc}^{\mathbf{p}}(0,T;W_{loc}^{1,\mathbf{p}}(\Omega))\cap L_{loc}^{P}(\Omega_{T})}}$
is a local weak sub(super)-solution of (\ref{eq:equation}) in $\Omega_{T}$
and let $h\in(0,T)$. Then, the Steklov average $[u]_{h}$ satisfies
\begin{equation}
\int_{\mathcal{K}\times\{\tau\}}\left(\frac{\partial[u]_{h}}{\partial t}\cdot\varphi\,+\langle[D_{\xi}F(x,t,Du)]_{h},D\varphi\rangle\right)dx\,\leq\,(\geq)\,0\label{eq:Steklov}
\end{equation}
for every compact subset $\mathcal{K}$ of $\Omega$, for all $\tau\in(0,T-h]$
and all test functions 
\[
\varphi\,\in\,C_{loc}^{0}\left(0,T;L^{2}(\mathcal{K})\right)\cap L_{loc}^{\mathbf{p}}\left(0,T;W_{0}^{1,\mathbf{p}}(\mathcal{K})\right),\,\,\,\,\,\,\varphi\geq0\,.
\]
}

\section{Energy classes\label{sec:energy_estimate}}

\noindent $\hspace*{1em}$In this section, we establish some energy
estimates that will be useful for proving our results. Throughout
the sequel, $(x_{0},t_{0})\in\Omega_{T}$ and $\rho>0$ are such that
$Q_{\rho}(x_{0},t_{0})\Subset\Omega_{T}$, while $\zeta\in C^{\infty}(\mathbb{R}^{n+1})$
is a smooth cut-off function of the form 
\begin{equation}
\zeta(x,t):=\,\psi(t)\,\prod_{i=1}^{n}\zeta_{i}^{p_{i}}(x_{i})\,,\label{eq:cut-off}
\end{equation}
where $\psi\in C^{\infty}(\mathbb{R};[0,1])$ is a non-decreasing
map such that
\begin{equation}
\psi\equiv0\,\,\,\,\,\mathrm{on}\,\,\left(-\infty,t_{0}-\rho+\frac{\epsilon}{2}\right]\,,\,\,\,\,\,\,\,\,\psi\equiv1\,\,\,\,\,\mathrm{on}\,\,\,[t_{0}-\rho+\epsilon,\infty)\,,\,\,\,\,\,\mathrm{with\,\,\,}\epsilon\in(0,\rho)\,,\label{eq:psi}
\end{equation}
while the functions $\zeta_{i}$ satisfy 
\begin{equation}
\zeta_{i}\,\in\,C_{0}^{\infty}(\mathbb{R};[0,1]),\,\,\,\,\,\,\,\,\mathrm{supp}\,\zeta_{i}\subset(x_{0,i}-\rho^{\frac{1}{p_{i}}},x_{0,i}+\rho^{\frac{1}{p_{i}}})\,\,\,\,\,\,\,\,\mathrm{for\,\,every}\,\,i\in\{1,\ldots,n\}.\label{eq:zeta_i}
\end{equation}
As an immediate consequence of these conditions, we have that
\[
\zeta\equiv0\,\,\,\,\,\,\mathrm{on\,\,the\,\,parabolic\,\,boundary\,\,of\,\,}Q_{\rho}(x_{0},t_{0})\,.
\]

\begin{prop}
\noindent \label{prop:PropEnergy}Let $n\geq2$ and $p_{1},\ldots,p_{n}>1$.
Assume that $(\ref{eq:coeff})$ holds and that $u$ is a local weak
subsolution of $(\ref{eq:equation})$ in the sense of Definition \ref{def:weaksoldef}.
Then there exists a positive constant $\mathcal{C}$, depending only
on $\Lambda$ and $\max\,\{p_{1},\ldots,p_{n}\}$, such that for every
$k\in\mathbb{R}$ we have\begin{align}\label{eq:energy_est}
&\underset{t_{0}-\rho\,<\,\tau\,<\,t_{0}}{\sup}\,\int_{K_{\rho}(x_{0})}(u-k)_{+}^{2}\,\zeta(x,\tau)\,dx\,+\,\sum_{i=1}^{n}\iint_{Q_{\rho}(x_{0},t_{0})}(\vert \partial_{x_{i}}u\vert-\delta_{i})_{+}^{p_{i}}\,\zeta\,\mathds{1}_{\{u\,>\,k\}}\,dx\,dt\nonumber\\
&\,\,\,\,\,\,\,\leq\,\mathcal{C}\iint_{Q_{\rho}(x_{0},t_{0})}(u-k)_{+}^{2}\,\partial_t\zeta\,dx\,dt\,+\,\mathcal{C}\,\sum_{i=1}^{n}\iint_{Q_{\rho}(x_{0},t_{0})}(u-k)_{+}^{p_{i}}\,\vert\partial_{x_{i}}\zeta^{\frac{1}{p_{i}}}\vert^{p_{i}}\,dx\,dt\,.
\end{align}Similarly, if $u$ is a local weak supersolution of $(\ref{eq:equation})$,
then for every $k\in\mathbb{R}$ we have\begin{align}\label{eq:energy_est02}
&\underset{t_{0}-\rho\,<\,\tau\,<\,t_{0}}{\sup}\,\int_{K_{\rho}(x_{0})}(u-k)_{-}^{2}\,\zeta(x,\tau)\,dx\,+\,\sum_{i=1}^{n}\iint_{Q_{\rho}(x_{0},t_{0})}(\vert \partial_{x_{i}}u\vert-\delta_{i})_{+}^{p_{i}}\,\zeta\,\mathds{1}_{\{u\,<\,k\}}\,dx\,dt\nonumber\\
&\,\,\,\,\,\,\,\leq\,\mathcal{C}\iint_{Q_{\rho}(x_{0},t_{0})}(u-k)_{-}^{2}\,\partial_t\zeta\,dx\,dt\,+\,\mathcal{C}\,\sum_{i=1}^{n}\iint_{Q_{\rho}(x_{0},t_{0})}(u-k)_{-}^{p_{i}}\,\vert\partial_{x_{i}}\zeta^{\frac{1}{p_{i}}}\vert^{p_{i}}\,dx\,dt\,.
\end{align}
\end{prop}

\noindent \begin{proof}[\bfseries{Proof}]Let $u$ be a local weak
subsolution of $(\ref{eq:equation})$. After a translation, we may
assume that $(x_{0},t_{0})=(0,0)$. Hence, it suffices to prove \eqref{eq:energy_est}
for the cylinder $Q_{\rho}$. In (\ref{eq:Steklov}) we take the test
functions 
\[
\varphi=([u]_{h}-k)_{+}\,\zeta
\]
and integrate with respect to time over $(-\rho,\tau)$, with $\tau\in(-\rho,0)$.
We thus obtain 
\begin{equation}
\int_{-\rho}^{\tau}\int_{K_{\rho}}\frac{\partial[u]_{h}}{\partial t}\,([u]_{h}-k)_{+}\,\zeta\,dx\,dt\,+\iint_{Q^{\tau}}\langle[A(x,t,Du)]_{h},D[([u]_{h}-k)_{+}\,\zeta]\rangle\,dx\,dt\,\leq\,0\,,\label{eq:weak1}
\end{equation}
where, for convenience of notation, we set 
\[
Q^{\tau}:=\,K_{\rho}\times(-\rho,\tau)\,\,\,\,\,\,\,\,\,\,\,\,\,\,\,\,\mathrm{and}\,\,\,\,\,\,\,\,\,\,\,\,\,\,\,A(x,t,\eta):=\,D_{\xi}F(x,t,\eta)\,,\,\,\,\,\,\eta\in\mathbb{R}^{n},
\]
with $F:\Omega_{T}\times\mathbb{R}^{n}\to\mathbb{R}$ denoting the
function defined in (\ref{eq:F}). The first term in (\ref{eq:weak1})
can be rewritten as
\[
\int_{-\rho}^{\tau}\int_{K_{\rho}}\frac{\partial[u]_{h}}{\partial t}\,([u]_{h}-k)_{+}\,\zeta\,dx\,dt\,=\,\frac{1}{2}\int_{-\rho}^{\tau}\int_{K_{\rho}}\frac{\partial([u]_{h}-k)_{+}^{2}}{\partial t}\,\zeta\,dx\,dt\,.
\]
Therefore, integrating by parts, using that $\zeta\equiv0$ on $\partial_{\mathrm{par}}Q_{\rho}$
and letting $h\to0$, by Lemma \ref{lem:Stek} we have\begin{align}\label{eq:limite1}
&\int_{-\rho}^{\tau}\int_{K_{\rho}}\frac{\partial[u]_{h}}{\partial t}\,([u]_{h}-k)_{+}\,\zeta\,dx\,dt\,\longrightarrow\nonumber\\
&\,\,\,\,\,\,\,\frac{1}{2}\int_{K_{\rho}}(u-k)_{+}^{2}\,\zeta(x,\tau)\,dx\,-\,\frac{1}{2}\iint_{Q^{\tau}}(u-k)_{+}^{2}\,\partial_t\zeta\,dx\,dt\,.
\end{align} Now observe that $\vert A_{i}(x,t,Du)\vert\leq\Lambda\,\vert\partial_{x_{i}}u\vert^{p_{i}-1}$
for any $i\in\{1,\ldots,n\}$. Then, taking the limit as $h\to0$
in the second term of (\ref{eq:weak1}), we can apply Lemma \ref{lem:Stek}
again. Thus we get\begin{align}\label{eq:limite2}
&\iint_{Q^{\tau}}\langle[A(x,t,Du)]_{h},D[([u]_{h}-k)_{+}\,\zeta]\rangle\,dx\,dt\,\longrightarrow\nonumber\\
&\,\,\,\,\,\,\,\iint_{Q^{\tau}}\langle A(x,t,Du),D(u-k)_{+}\rangle\,\zeta\,dx\,dt\,+\iint_{Q^{\tau}}\langle A(x,t,Du),D\zeta\rangle\,(u-k)_{+}\,dx\,dt\,.
\end{align}From (\ref{eq:weak1})$-$\eqref{eq:limite2}, we then obtain\begin{align}\label{eq:uguaglianza}
&\frac{1}{2}\int_{K_{\rho}}(u-k)_{+}^{2}\,\zeta(x,\tau)\,dx\,+\iint_{Q^{\tau}}\langle A(x,t,Du),D(u-k)_{+}\rangle\,\zeta\,dx\,dt\nonumber\\
&\,\,\,\,\,\,\,\leq\,\frac{1}{2}\iint_{Q^{\tau}}(u-k)_{+}^{2}\,\partial_t\zeta\,dx\,dt\,-\iint_{Q^{\tau}}\langle A(x,t,Du),D\zeta\rangle\,(u-k)_{+}\,dx\,dt\,.
\end{align}We now estimate\begin{align*}
\iint_{Q^{\tau}}\langle A(x,t,Du),D(u-k)_{+}\rangle\,\zeta\,dx\,dt\,&=\,\sum_{i=1}^{n}\iint_{Q^{\tau}\,\cap\,\{u\,>\,k\}}a_{i}(x,t)\,(\vert \partial_{x_{i}}u\vert-\delta_{i})_{+}^{p_{i}-1}\,\vert \partial_{x_{i}}u\vert\,\zeta\,dx\,dt\\
&\geq\,\Lambda^{-1}\,\sum_{i=1}^{n}\iint_{Q^{\tau}\,\cap\,\{u\,>\,k\}}(\vert \partial_{x_{i}}u\vert-\delta_{i})_{+}^{p_{i}}\,\zeta\,dx\,dt\,.
\end{align*}To deal with the last term in \eqref{eq:uguaglianza}, we set, for
each $i\in\{1,\ldots,n\}$, 
\[
\hat{\zeta}_{i}:=\prod_{j\neq i}\zeta_{j}^{p_{j}}(x_{j})\,.
\]
Using the definition of $\hat{\zeta}_{i}$, we get 
\[
\psi\,\hat{\zeta}_{i}\,\zeta_{i}^{p_{i}}\,=\,\zeta\,\,\,\,\,\,\,\,\,\,\,\,\,\,\,\mathrm{and}\,\,\,\,\,\,\,\,\,\,\,\,\,\,\,\psi\,\hat{\zeta}_{i}\,\vert\partial_{x_{i}}\zeta_{i}\vert^{p_{i}}\,=\,\vert\partial_{x_{i}}\zeta^{\frac{1}{p_{i}}}\vert^{p_{i}}\,.
\]
Therefore, applying Young's inequality with $\varepsilon_{i}>0$,
we have\begin{align*}
&-\iint_{Q^{\tau}}\langle A(x,t,Du),D\zeta\rangle\,(u-k)_{+}\,dx\,dt\\
&\,\,\,\leq\,\sum_{i=1}^{n}\,p_{i}\iint_{Q^{\tau}}a_{i}(x,t)\,(\vert \partial_{x_{i}}u\vert-\delta_{i})_{+}^{p_{i}-1}\,\psi\,\hat{\zeta}_{i}\,\zeta_{i}^{p_{i}-1}\,\vert\partial_{x_{i}}\zeta_{i}\vert\,(u-k)_{+}\,dx\,dt\\
&\,\,\,\leq\,\Lambda\,\sum_{i=1}^{n}\varepsilon_{i}(p_{i}-1)\iint_{Q^{\tau}\,\cap\,\{u\,>\,k\}}(\vert \partial_{x_{i}}u\vert-\delta_{i})_{+}^{p_{i}}\,\zeta\,dx\,dt\,+\,\Lambda\,\sum_{i=1}^{n}\frac{1}{\varepsilon_{i}^{p_{i}-1}}\iint_{Q^{\tau}}(u-k)_{+}^{p_{i}}\,\vert\partial_{x_{i}}\zeta^{\frac{1}{p_{i}}}\vert^{p_{i}}\,dx\,dt\,.
\end{align*}Choosing $\varepsilon_{i}=[2\,(p_{i}-1)\,\Lambda^{2}]^{-1}$ for every
$i\in\{1,...,n\}$, and collecting the three previous estimates, we
obtain \begin{align*}
&\int_{K_{\rho}}(u-k)_{+}^{2}\,\zeta(x,\tau)\,dx\,+\,\sum_{i=1}^{n}\iint_{Q^{\tau}\,\cap\,\{u\,>\,k\}}(\vert \partial_{x_{i}}u\vert-\delta_{i})_{+}^{p_{i}}\,\zeta\,dx\,dt\\
&\,\,\,\,\,\,\,\leq\,\mathcal{C}\iint_{Q^{\tau}}(u-k)_{+}^{2}\,\partial_t\zeta\,dx\,dt\,+\,\mathcal{C}\,\sum_{i=1}^{n}\iint_{Q^{\tau}}(u-k)_{+}^{p_{i}}\,\vert\partial_{x_{i}}\zeta^{\frac{1}{p_{i}}}\vert^{p_{i}}\,dx\,dt\,,
\end{align*}where $\mathcal{C}$ is a positive constant depending only on $\Lambda$
and $\max\,\{p_{1},\ldots,p_{n}\}$. Recalling that $\tau\in(-\rho,0)$
is arbitrary, from the above inequality we get\begin{align*}
&\underset{-\rho\,<\,\tau\,<\,0}{\sup}\,\int_{K_{\rho}}(u-k)_{+}^{2}\,\zeta(x,\tau)\,dx\,+\,\sum_{i=1}^{n}\iint_{Q_{\rho}}(\vert \partial_{x_{i}}u\vert-\delta_{i})_{+}^{p_{i}}\,\zeta\,\mathds{1}_{\{u\,>\,k\}}\,dx\,dt\\
&\,\,\,\,\,\,\,\leq\,\mathcal{C}\iint_{Q_{\rho}}(u-k)_{+}^{2}\,\partial_t\zeta\,dx\,dt\,+\,\mathcal{C}\,\sum_{i=1}^{n}\iint_{Q_{\rho}}(u-k)_{+}^{p_{i}}\,\vert\partial_{x_{i}}\zeta^{\frac{1}{p_{i}}}\vert^{p_{i}}\,dx\,dt\,.
\end{align*}We have thus proved estimate \eqref{eq:energy_est}.\\
Now observe that if $u$ is a local weak supersolution of (\ref{eq:equation}),
then $-u$ is a local weak subsolution. Hence, applying \eqref{eq:energy_est}
with $-u$ and $-k$ in place of $u$ and $k$, respectively, and
using the identity $(-v)_{+}=v_{-}$, we immediately obtain inequality
\eqref{eq:energy_est02}.\end{proof}
\begin{defn}[\textbf{De Giorgi classes $\mathcal{DG}^{+}(\mathbf{p},\{\delta_{i}\},\Omega_{T},\mathcal{C})$}]
\noindent  A measurable function $u:\Omega_{T}\to\mathbb{R}$ belongs
to the class $\mathcal{DG}^{+}(\mathbf{p},\{\delta_{i}\},\Omega_{T},\mathcal{C})$
if 
\[
{\color{red}{\normalcolor u\,\in\,C_{loc}^{0}(0,T;L_{loc}^{2}(\Omega))\cap L_{loc}^{\mathbf{p}}(0,T;W_{loc}^{1,\mathbf{p}}(\Omega))\cap L_{loc}^{P}(\Omega_{T})}}
\]
and for all $k\in\mathbb{R}$ and all $\zeta$ as above, the function
$u$ satisfies inequality \eqref{eq:energy_est}. 
\end{defn}

\begin{defn}[\textbf{De Giorgi classes $\mathcal{DG}^{-}(\mathbf{p},\{\delta_{i}\},\Omega_{T},\mathcal{C})$}]
\noindent  A measurable function $u:\Omega_{T}\to\mathbb{R}$ belongs
to the class $\mathcal{DG}^{-}(\mathbf{p},\{\delta_{i}\},\Omega_{T},\mathcal{C})$
if 
\[
{\color{red}{\normalcolor u\,\in\,C_{loc}^{0}(0,T;L_{loc}^{2}(\Omega))\cap L_{loc}^{\mathbf{p}}(0,T;W_{loc}^{1,\mathbf{p}}(\Omega))\cap L_{loc}^{P}(\Omega_{T})}}
\]
and for all $k\in\mathbb{R}$ and all $\zeta$ as above, the function
$u$ satisfies inequality \eqref{eq:energy_est02}. 
\end{defn}

\begin{defn}[\textbf{De Giorgi classes $\mathcal{DG}(\mathbf{p},\{\delta_{i}\},\Omega_{T},\mathcal{C})$}]
\label{def:DeGiorgi}A measurable function $u:\Omega_{T}\to\mathbb{R}$
belongs to the class $\mathcal{DG}(\mathbf{p},\{\delta_{i}\},\Omega_{T},\mathcal{C})$
if 
\[
{\color{red}{\normalcolor u\,\in\,C_{loc}^{0}(0,T;L_{loc}^{2}(\Omega))\cap L_{loc}^{\mathbf{p}}(0,T;W_{loc}^{1,\mathbf{p}}(\Omega))\cap L_{loc}^{P}(\Omega_{T})}}
\]
and for all $k\in\mathbb{R}$ and all $\zeta$ as above, the function
$u$ satisfies both inequalities \eqref{eq:energy_est} and \eqref{eq:energy_est02}. 
\end{defn}

\noindent \begin{brem}\label{def:ossDG}By Proposition \ref{prop:PropEnergy},
under assumption (\ref{eq:coeff}), local weak subsolutions (resp.
supersolutions) to (\ref{eq:equation}) in the sense of Definition
\ref{def:weaksoldef} belong to $\mathcal{DG}^{+}(\mathbf{p},\{\delta_{i}\},\Omega_{T},\mathcal{C})$
(resp. $\mathcal{DG}^{-}(\mathbf{p},\{\delta_{i}\},\Omega_{T},\mathcal{C})$).
In particular, local weak solutions to (\ref{eq:equation}) in the
sense of Definition \ref{def:locweaksol2} belong to $\mathcal{DG}(\mathbf{p},\{\delta_{i}\},\Omega_{T},\mathcal{C})$.\end{brem}\medskip{}

\noindent $\hspace*{1em}$Starting from the above definitions, in
the following sections we prove local boundedness and the existence
of semicontinuous representatives for functions in $\mathcal{DG}^{\pm}(\mathbf{p},\{\delta_{i}\},\Omega_{T},\mathcal{C})$,
under suitable assumptions on the exponents $p_{i}$ appearing in
(\ref{eq:equation}).

\section{Local boundedness\label{sec:boundedness}}

\noindent $\hspace*{1em}$In this section we prove the local boundedness
of local weak solutions to equation (\ref{eq:equation}). Our proofs
are based on De Giorgi-type iterations combining the energy estimates
established in Proposition \ref{prop:PropEnergy} with the Sobolev
embedding from Remark \ref{SobEmb}. Throughout the section, we will
use the space-time cylinders defined in Section \ref{sec:prelim},
which turn out to be convenient in our setting.\\
$\hspace*{1em}$As mentioned in the introduction, there are two ranges
for $\overline{p}$ that require somewhat different arguments, and
in the range corresponding to small values of $\overline{p}$ we also
require some extra integrability of the local weak solutions. For
clarity, the two cases have been treated in separate subsections.

\subsection{The case $\overline{p}>\frac{2n}{n+2}$ }

\noindent $\hspace*{1em}$In this section we focus on the case in
which $\overline{p}$ satisfies the following lower bound: 
\begin{equation}
\overline{p}\,>\,\frac{2n}{n+2}\,.\label{eq:supercritical}
\end{equation}
The proof of local boundedness in Theorem \ref{thm:main} below is
somewhat different in the cases $\mathcal{P}\geq2$ and $\mathcal{P}<2$,
where $\mathcal{P}:=\max\,\{p_{1},\ldots,p_{n}\}$. In fact, in the
former case the condition (\ref{eq:supercritical}) is not explicitly
used in the argument. However, in this particular case, (\ref{eq:supercritical})
must necessarily be true due to the upper bound for the parameters
$p_{i}$ in (\ref{eq:range}).
\begin{thm}
\noindent \label{thm:main}Let $u\in\mathcal{DG}^{+}(\mathbf{p},\{\delta_{i}\},\Omega_{T},\mathcal{C})$,
where $\mathbf{p}=(p_{1},\ldots,p_{n})$ satisfies $(\ref{eq:range})$,
and assume that $(\ref{eq:supercritical})$ holds. Then $u$ is locally
bounded from above and, for every cylinder $Q_{\rho}(x_{0},t_{0})\Subset\Omega_{T}$
and every $\sigma\in(0,1)$, we have the explicit bound
\begin{equation}
\underset{Q_{\sigma\rho}(x_{0},t_{0})}{\mathrm{ess}\,\sup}\,u\,\leq\,\max\left\{ 1,\,\rho^{1/P},\,C\left[(1-\sigma)^{-\,\frac{\mathcal{P}\,(n\,+\,\overline{p})}{\overline{p}}}\tiltfiint_{Q_{\rho}(x_{0},t_{0})}u_{+}^{P}\,dx\,dt\right]^{\frac{\overline{p}}{\overline{p}\,(n\,+\,2)\,-\,nP}}\right\} ,\label{eq:supremum}
\end{equation}
where $\mathcal{P}:=\max\,\{p_{1},\ldots,p_{n}\}$, $P:=\max\,\{2,\mathcal{P}\}$
and $C$ is a positive constant depending only on $n$, $\Lambda$,
$\mathbf{p}$ and $\max\,\{\delta_{1},\ldots,\delta_{n}\}$. Similarly,
if $u\in\mathcal{DG}^{-}(\mathbf{p},\{\delta_{i}\},\Omega_{T},\mathcal{C})$
under assumptions $(\ref{eq:range})$ and $(\ref{eq:supercritical})$,
then $u$ is locally bounded from below and, for every $Q_{\rho}(x_{0},t_{0})\Subset\Omega_{T}$
and every $\sigma\in(0,1)$, we have 
\begin{equation}
\underset{Q_{\sigma\rho}(x_{0},t_{0})}{\mathrm{ess}\,\inf}\,u\,\geq\,-\,\max\left\{ 1,\,\rho^{1/P},\,C\left[(1-\sigma)^{-\,\frac{\mathcal{P}\,(n\,+\,\overline{p})}{\overline{p}}}\tiltfiint_{Q_{\rho}(x_{0},t_{0})}u_{-}^{P}\,dx\,dt\right]^{\frac{\overline{p}}{\overline{p}\,(n\,+\,2)\,-\,nP}}\right\} .\label{eq:infimum}
\end{equation}
\end{thm}

\noindent \begin{brem}\label{thm:Rk1}In view of Remark \ref{def:ossDG}
and Theorem \ref{thm:main}, under assumptions (\ref{eq:coeff}),
(\ref{eq:range}) and (\ref{eq:supercritical}), local weak subsolutions
to (\ref{eq:equation}) are locally bounded from above in $\Omega_{T}$,
whereas local weak supersolutions are locally bounded from below.\end{brem}\vspace{-2mm}

\noindent \begin{proof}[\bfseries{Proof of Theorem~\ref{thm:main}}]Let
us first assume that \textit{$u\in\mathcal{DG}^{+}(\mathbf{p},\{\delta_{i}\},\Omega_{T},\mathcal{C})$.}
After a translation, we may suppose that $(x_{0},t_{0})$ coincides
with the origin. For a fixed $\sigma\in(0,1)$, we consider the sequence
\begin{equation}
\rho_{j}:=\,\sigma\rho\,+\,\frac{(1-\sigma)}{2^{j}}\,\rho\,,\,\,\,\,\,\,\,\,\,\,j\in\mathbb{N}_{0}\,,\label{eq:radii}
\end{equation}
and the corresponding cylinders $\mathcal{Q}_{j}:=Q_{\rho_{j}}$.
From the definitions it follows that $\mathcal{Q}_{0}=Q_{\rho}$.
We also consider the family of boxes 
\[
\widetilde{\mathcal{Q}}_{j}:=\,Q_{\tilde{\rho}_{j}}\,,
\]
where, for $j\in\mathbb{N}_{0}$, 
\begin{equation}
\tilde{\rho}_{j}:=\,\frac{\rho_{j}+\rho_{j+1}}{2}\,=\,\sigma\rho\,+\,\frac{3(1-\sigma)}{2^{j+2}}\,\rho\,.\label{eq:radii2}
\end{equation}
For these boxes, we have the inclusions
\[
\mathcal{Q}_{j+1}\subset\widetilde{\mathcal{Q}}_{j}\subset\mathcal{Q}_{j}\,,\,\,\,\,\,\,\,\,\,\,j\in\mathbb{N}_{0}\,.
\]
We now introduce the sequence of increasing levels 
\begin{equation}
k_{j}:=\,k-\,\frac{k}{2^{j}}\,,\label{eq:levels}
\end{equation}
where $k\geq1$ is a number to be chosen later. To move forward, we
choose cut-off functions $\tilde{\zeta}_{i,j}\in C_{0}^{\infty}((-\rho_{j}^{\frac{1}{p_{i}}},\rho_{j}^{\frac{1}{p_{i}}});[0,1])$
such that 
\[
\tilde{\zeta}_{i,j}\equiv1\,\,\,\,\,\mathrm{on}\,\,\,(-\rho_{j+1}^{\frac{1}{p_{i}}},\rho_{j+1}^{\frac{1}{p_{i}}})\,\,\,\,\,\,\,\,\,\,\,\,\,\,\,\mathrm{and}\,\,\,\,\,\,\,\,\,\,\,\,\,\,\,\vert\tilde{\zeta}'_{i,j}\vert^{p_{i}}\,\leq\,\frac{c\,2^{j\mathcal{P}}}{(1-\sigma)^{\mathcal{P}}\rho}\,,
\]
where $c=c(\mathcal{P})>0$. Furthermore, we define
\[
\eta_{j}(x):=\,\prod_{i=1}^{n}\tilde{\zeta}_{i,j}^{p_{i}}(x_{i})\,.
\]

\noindent Similarly, we take $\psi_{j}\in C^{\infty}(\mathbb{R};[0,1])$
such that 
\[
\psi_{j}\equiv0\,\,\,\,\,\mathrm{on}\,\,(-\infty,-\tilde{\rho}_{j}]\,,\,\,\,\,\,\,\,\,\psi_{j}\equiv1\,\,\,\,\,\mathrm{on}\,\,\,[-\rho_{j+1},\infty)\,\,\,\,\,\,\,\,\,\,\,\mathrm{and}\,\,\,\,\,\,\,\,\,\,\,0\,\leq\,\psi'_{j}\,\leq\,\frac{c\,2^{j}}{(1-\sigma)\,\rho}\,.
\]
It follows that the smooth cut-off function 
\[
\tilde{\zeta}_{j}(x,t):=\,\psi_{j}(t)\,\eta_{j}(x)
\]

\noindent satisfies the following properties:
\begin{equation}
\begin{cases}
\begin{array}{c}
{\displaystyle 0\leq\tilde{\zeta}_{j}\leq1\,,\,\,\,\,\,\,\,\,\,\,\tilde{\zeta}_{j}\equiv1\,\,\,\,\,\,\mathrm{in}\,\,\,\mathcal{Q}_{j+1}\,,\,\,\,\,\,\,\,\,\,\,\tilde{\zeta}_{j}\equiv0\,\,\,\,\,\,\mathrm{on}\,\,\,\partial_{\mathrm{par}}\mathcal{Q}_{j}\,,\,\,\,\,\,\,\,\,\,\,\,\,\,\,\,\,\,\,\,\,\,\,\,\,\,\,\,\,\,\,\,\,\,\,\,\,\,\,\,\,\,\,\,\,\,\,\,\,\,\,\,\,\,\,}\vspace{4mm}\\
{\displaystyle 0\,\leq\,\partial_{t}\tilde{\zeta}_{j}\,\leq\,\frac{c\,2^{j}}{(1-\sigma)\,\rho}\,\,,\,\,\,\,\,\,\,\,\,\,\vert\partial_{x_{i}}\tilde{\zeta}_{j}^{\frac{1}{p_{i}}}\vert^{p_{i}}\,\leq\,\frac{c\,2^{j\mathcal{P}}}{(1-\sigma)^{\mathcal{P}}\rho}\,\,\,\,\,\,\,\,\mathrm{for\,\,every}\,\,i\in\{1,\ldots,n\}.}
\end{array}\end{cases}\label{eq:properties}
\end{equation}
Details on the explicit construction of such a cut-off function are
provided in the \hyperref[sec:appendice]{Appendix}. Now we rewrite
estimate \eqref{eq:energy_est} with $\rho_{j}$, $k_{j+1}$ and $\tilde{\zeta}_{j}$
in place of $\rho$, $k$ and $\zeta,$ respectively. We thus obtain\begin{align}\label{eq:Caccioppoli1}
&\underset{-\rho_{j}\,<\,\tau\,<\,0}{\sup}\,\int_{K_{\rho_{j}}}(u-k_{j+1})_{+}^{2}\,\tilde{\zeta}_{j}(x,\tau)\,dx\,+\,\sum_{i=1}^{n}\iint_{\mathcal{Q}_{j}\,\cap\,\{u\,>\,k_{j+1}\}}(\vert \partial_{x_{i}}u\vert-\delta_{i})_{+}^{p_{i}}\,\tilde{\zeta}_{j}\,dx\,dt\nonumber\\
&\,\,\,\,\,\,\,\leq\,\frac{C_0\,2^{j}}{(1-\sigma)\,\rho}\iint_{\mathcal{Q}_{j}}(u-k_{j+1})_{+}^{2}\,dx\,dt\,+\,\frac{C_0\,2^{j\mathcal{P}}}{(1-\sigma)^{\mathcal{P}}\rho}\,\sum_{i=1}^{n}\iint_{\mathcal{Q}_{j}}(u-k_{j+1})_{+}^{p_{i}}\,dx\,dt\,,
\end{align}where $C_{0}$ is a positive constant depending only on $\Lambda$
and $\mathcal{P}$. In addition, we set 
\[
q:=\,\overline{p}\left(1+\frac{2}{n}\right).
\]
In the case $\mathcal{P}\geq2$, the upper bound of the exponents
$p_{i}$ in (\ref{eq:range}) ensures that $\frac{q}{P}>1$, where
$P:=\max\,\{2,\mathcal{P}\}$. In the case $\mathcal{P}<2$, we can
use the lower bound (\ref{eq:supercritical}) to deduce that $\frac{q}{P}>1$.\\
$\hspace*{1em}$At this point of the proof, we define the sequence
\begin{equation}
Y_{j}:=\iint_{\mathcal{Q}_{j}}(u-k_{j})_{+}^{P}\,dx\,dt\,,\,\,\,\,\,\,\,\,\,\,j\in\mathbb{N}_{0}\,.\label{eq:Y_j}
\end{equation}
Note that $Y_{j}$ is finite for every $j\in\mathbb{N}_{0}$, since
$u\in L_{loc}^{P}(\Omega_{T})$. We may use Hölder's inequality to
estimate 
\begin{equation}
Y_{j+1}\,\leq\left[\iint_{\mathcal{Q}_{j+1}}(u-k_{j+1})_{+}^{q}\,dx\,dt\right]^{\frac{P}{q}}\vert A_{j+1}\vert^{1\,-\,\frac{P}{q}}\,,\label{eq:Yj+1}
\end{equation}
where we have set 
\begin{equation}
A_{j+1}:=\,\left\{ (x,t)\in\mathcal{Q}_{j}:u(x,t)>k_{j+1}\right\} .\label{eq:superlevel}
\end{equation}
In the following calculation, we estimate the integral in (\ref{eq:Yj+1})
by applying Hölder's inequality with $H:=n/(n-\overline{p})$ in the
integral over space. Then, we estimate one of the integrals in space
by its supremum for all times in the interval $(-\rho_{j+1},0)$,
and introduce the functions $\eta_{j}$ and $\psi_{j}$ while expanding
the set of integration, which allows us to use (\ref{eq:Troisi}).
All in all, the calculation takes the form\begin{align}\label{eq:big_calculation}
&\iint_{\mathcal{Q}_{j+1}}(u-k_{j+1})_{+}^{q}\,dx\,dt\,\le\int_{-\rho_{j+1}}^{0}\left[\int_{K_{\rho_{j+1}}}(u-k_{j+1})_{+}^{\frac{2\overline{p}}{n}\,H'}\,dx\right]^{\frac{1}{H'}}\left[\int_{K_{\rho_{j+1}}}(u-k_{j+1})_{+}^{\overline{p}H}\,dx\right]^{\frac{1}{H}}dt\nonumber\\
&\,\,\,\,\,\,\,=\int_{-\rho_{j+1}}^{0}\left[\int_{K_{\rho_{j+1}}}(u-k_{j+1})_{+}^{2}\,dx\right]^{\frac{\overline{p}}{n}}\left[\int_{K_{\rho_{j+1}}}(u-k_{j+1})_{+}^{\overline{p}^{*}}\,dx\right]^{\frac{\overline{p}}{\overline{p}^{*}}}dt\nonumber\\
&\,\,\,\,\,\,\,\leq\left[\underset{-\rho_{j+1}\,<\,\tau\,<\,0}{\sup}\int_{K_{\rho_{j+1}}}(u(x,\tau)-k_{j+1})_{+}^{2}\,dx\right]^{\frac{\overline{p}}{n}}\int_{-\rho_{j+1}}^{0}\left[\int_{K_{\rho_{j+1}}}(u-k_{j+1})_{+}^{\overline{p}^{*}}\,dx\right]^{\frac{\overline{p}}{\overline{p}^{*}}}dt\nonumber\\
&\,\,\,\,\,\,\,\leq\left[\underset{-\rho_{j}\,<\,\tau\,<\,0}{\sup}\int_{K_{\rho_{j}}}(u-k_{j+1})_{+}^{2}\,\tilde{\zeta}_{j}(x,\tau)\,dx\right]^{\frac{\overline{p}}{n}}\int_{-\rho_{j}}^{0}\psi_{j}(t)\left[\int_{K_{\rho_{j}}}[(u-k_{j+1})_{+}\,\eta_{j}]^{\overline{p}^{*}}\,dx\right]^{\frac{\overline{p}}{\overline{p}^{*}}}dt\nonumber\\
&\,\,\,\,\,\,\,\leq\,c_{0}\left[\underset{-\rho_{j}\,<\,\tau\,<\,0}{\sup}\int_{K_{\rho_{j}}}(u-k_{j+1})_{+}^{2}\,\tilde{\zeta}_{j}(x,\tau)\,dx\right]^{\frac{\overline{p}}{n}}\left[\sum_{i=1}^{n}\iint_{\mathcal{Q}_{j}}\vert\partial_{x_{i}}[(u-k_{j+1})_{+}\,\eta_{j}]\vert^{p_{i}}\,\psi_{j}\,dx\,dt\right],
\end{align} where $c_{0}=c_{0}(n,\mathbf{p})>0$. Now we estimate the last term
in \eqref{eq:big_calculation}. Setting 
\[
\delta:=\,\max\,\{\delta_{i}:i=1,\ldots,n\},
\]
applying the definition of $A_{j+1}$ in (\ref{eq:superlevel}), and
using the bounds for the functions $\eta_{j}$, $\psi_{j}$ and their
derivatives, we obtain\begin{align}\label{eq:big2}
&\sum_{i=1}^{n}\iint_{\mathcal{Q}_{j}}\vert\partial_{x_{i}}[(u-k_{j+1})_{+}\,\eta_{j}]\vert^{p_{i}}\,\psi_{j}\,dx\,dt\,=\,\sum_{i=1}^{n}\iint_{A_{j+1}}\vert (\partial_{x_{i}}u)\,\eta_{j}\,+\,(u-k_{j+1})_{+}\,\partial_{x_{i}}\eta_{j}\vert^{p_{i}}\,\psi_{j}\,dx\,dt\nonumber\\
&\leq\,c\left[\sum_{i=1}^{n}\iint_{A_{j+1}}\vert \partial_{x_{i}}u\vert^{p_{i}}\,\tilde{\zeta}_{j}\,dx\,dt\,+\,\sum_{i=1}^{n}\iint_{A_{j+1}}(u-k_{j+1})_{+}^{p_{i}}\,\vert\partial_{x_{i}}\eta_{j}\vert^{p_{i}}\,\psi_{j}\,dx\,dt\right]\nonumber\\
&\leq\,c\left[\sum_{i=1}^{n}\iint_{A_{j+1}}[(\vert \partial_{x_{i}}u\vert-\delta_{i})_{+}+\delta_{i}]^{p_{i}}\,\tilde{\zeta}_{j}\,dx\,dt\,+\,\frac{2^{j\mathcal{P}}}{(1-\sigma)^{\mathcal{P}}\,\rho}\,\sum_{i=1}^{n}\iint_{\mathcal{Q}_{j}}(u-k_{j+1})_{+}^{p_{i}}\,dx\,dt\right]\nonumber\\
&\leq\,c\left[\sum_{i=1}^{n}\iint_{A_{j+1}}(\vert \partial_{x_{i}}u\vert-\delta_{i})_{+}^{p_{i}}\,\tilde{\zeta}_{j}\,dx\,dt\,+\,n\,(\delta+1)^{\mathcal{P}}\,\vert A_{j+1}\vert\,+\,\frac{2^{j\mathcal{P}}}{(1-\sigma)^{\mathcal{P}}\,\rho}\,\sum_{i=1}^{n}\iint_{\mathcal{Q}_{j}}(u-k_{j+1})_{+}^{p_{i}}\,dx\,dt\right],
\end{align}where $c=c(\mathcal{P})>0$. Combining \eqref{eq:big_calculation}
and \eqref{eq:big2}, majorizing the supremum in \eqref{eq:big_calculation}
and the final term in \eqref{eq:big2} by means of the energy estimate
\eqref{eq:Caccioppoli1}, and using that $\frac{1}{1-\sigma}>1$,
we ultimately get\begin{align}\label{eq:big_bis}
&\iint_{\mathcal{Q}_{j+1}}(u-k_{j+1})_{+}^{q}\,dx\,dt\nonumber\\
&\,\,\,\,\,\,\,\leq\,\frac{C_{1}\,2^{j\,\frac{\mathcal{P}\,(n\,+\,\overline{p})}{n}}}{(1-\sigma)^{\frac{\mathcal{P}\,(n\,+\,\overline{p})}{n}}}\left[\frac{1}{\rho}\iint_{\mathcal{Q}_{j}}(u-k_{j+1})_{+}^{2}\,dx\,dt\,+\,\frac{1}{\rho}\,\sum_{i=1}^{n}\iint_{\mathcal{Q}_{j}}(u-k_{j+1})_{+}^{p_{i}}\,dx\,dt\,+\,\vert A_{j+1}\vert\,\right]^{\frac{n\,+\,\overline{p}}{n}},
\end{align} for a positive constant $C_{1}$ depending only on $n$, $\Lambda$,
$\mathbf{p}$ and $\delta$. Invoking again the definition of $A_{j+1}$,
recalling that $k_{j}<k$, $k\geq1$ and $p_{i}\leq\mathcal{P}\leq P$,
and applying Young's inequality, we can handle the last sum as follows
\begin{align}\label{eq:big_bis002}
&\sum_{i=1}^{n}\iint_{\mathcal{Q}_{j}}(u-k_{j+1})_{+}^{p_{i}}\,dx\,dt\,\leq\,\sum_{i=1}^{n}\iint_{A_{j+1}}u^{p_{i}}\,dx\,dt\,\leq\,\sum_{i=1}^{n}\iint_{A_{j+1}}[(u-k_{j})_{+}+k_{j}]^{p_{i}}\,dx\,dt\nonumber\\
&\,\,\,\,\,\,\,\leq\,2^{\mathcal{P}-1}\left[\sum_{i=1}^{n}\iint_{A_{j+1}}(u-k_{j})_{+}^{p_{i}}\,dx\,dt\,+\,n\,k^{P}\,\vert A_{j+1}\vert\right]\nonumber\\
&\,\,\,\,\,\,\,\leq\,2^{\mathcal{P}-1}\,n\left[\iint_{A_{j+1}}(u-k_{j})_{+}^{P}\,dx\,dt\,+\,(1+k^{P})\,\vert A_{j+1}\vert\right]\nonumber\\
&\,\,\,\,\,\,\,\leq\,2^{\mathcal{P}}\,n\left[Y_{j}\,+\,k^{P}\,\vert A_{j+1}\vert\right],
\end{align}where, in the last line, we have used the definition of $Y_{j}$ in
(\ref{eq:Y_j}). Similarly, since $P\geq2$, the first integral on
the right-hand side of \eqref{eq:big_bis} can also be treated using
Young's inequality: 
\[
\iint_{\mathcal{Q}_{j}}(u-k_{j+1})_{+}^{2}\,dx\,dt\,\leq\iint_{A_{j+1}}(u-k_{j+1})_{+}^{P}\,dx\,dt\,+\,\vert A_{j+1}\vert\,\leq\,Y_{j}\,+\,\vert A_{j+1}\vert\,,
\]
where, in the last inequality, we have used the fact that $k_{j}<k_{j+1}$.
Collecting the last three estimates, redefining the constant $C_{1}$
and recalling that $k\geq1$, we obtain\begin{align*}
\iint_{\mathcal{Q}_{j+1}}(u-k_{j+1})_{+}^{q}\,dx\,dt\,&\leq\,\frac{C_{1}\,2^{j\,\frac{\mathcal{P}\,(n\,+\,\overline{p})}{n}}}{(1-\sigma)^{\frac{\mathcal{P}\,(n\,+\,\overline{p})}{n}}}\left[\frac{Y_{j}}{\rho}\,+\left(1+\,\frac{1+k^{P}}{\rho}\right)\vert A_{j+1}\vert\right]^{\frac{n\,+\,\overline{p}}{n}}\nonumber\\
&\leq\,\frac{C_{1}\,2^{j\,\frac{\mathcal{P}\,(n\,+\,\overline{p})}{n}}}{(1-\sigma)^{\frac{\mathcal{P}\,(n\,+\,\overline{p})}{n}}}\left[\frac{Y_{j}}{\rho}\,+\left(1+\,\frac{2\,k^{P}}{\rho}\right)\vert A_{j+1}\vert\right]^{\frac{n\,+\,\overline{p}}{n}}.
\end{align*}Without loss of generality, we can now assume that $k\geq\max\,\{1,\rho^{1/P}\}$.
It then follows from the previous inequality that
\begin{equation}
\iint_{\mathcal{Q}_{j+1}}(u-k_{j+1})_{+}^{q}\,dx\,dt\,\leq\,\frac{C_{1}\,2^{j\,\frac{\mathcal{P}\,(n\,+\,\overline{p})}{n}}}{(1-\sigma)^{\frac{\mathcal{P}\,(n\,+\,\overline{p})}{n}}\,\rho^{\frac{n\,+\,\overline{p}}{n}}}\left(Y_{j}\,+\,k^{P}\,\vert A_{j+1}\vert\right)^{\frac{n\,+\,\overline{p}}{n}},\label{eq:clef}
\end{equation}
up to a different constant $C_{1}=C_{1}(n,\Lambda,\mathbf{p},\delta)>0$.
At this point, we note that 
\begin{equation}
\vert A_{j+1}\vert\,=\,\iint_{A_{j+1}}\frac{(k_{j+1}-k_{j})^{P}}{(k_{j+1}-k_{j})^{P}}\,dx\,dt\,\leq\,\iint_{A_{j+1}}\frac{(u-k_{j})^{P}}{(k_{j+1}-k_{j})^{P}}\,dx\,dt\,\leq\,\frac{2^{(j+1)P}}{k^{P}}\,Y_{j}\,.\label{eq:measure}
\end{equation}
Combining this estimate with (\ref{eq:clef}) and using that $1\leq\mathcal{P}\leq P$,
we get\begin{align}\label{eq:clef02}
\iint_{\mathcal{Q}_{j+1}}(u-k_{j+1})_{+}^{q}\,dx\,dt\,&\leq\,\frac{C_{1}\,2^{j\,\frac{\mathcal{P}\,(n\,+\,\overline{p})}{n}}}{(1-\sigma)^{\frac{\mathcal{P}\,(n\,+\,\overline{p})}{n}}\,\rho^{\frac{n\,+\,\overline{p}}{n}}}\left(1+2^{(j+1)P}\right)^{\frac{n\,+\,\overline{p}}{n}}\,Y_{j}^{1\,+\,\frac{\overline{p}}{n}}\nonumber\\
&\leq\,\frac{C_{1}\,2^{j\,\frac{P\,(n\,+\,\overline{p})}{n}}\,2^{(j+2)\,\frac{P\,(n\,+\,\overline{p})}{n}}}{(1-\sigma)^{\frac{\mathcal{P}\,(n\,+\,\overline{p})}{n}}\,\rho^{\frac{n\,+\,\overline{p}}{n}}}\,\,Y_{j}^{1\,+\,\frac{\overline{p}}{n}}\nonumber\\
&\leq\,\frac{C_{1}\,4^{j\,\frac{P\,(n\,+\,\overline{p})}{n}}}{(1-\sigma)^{\frac{\mathcal{P}\,(n\,+\,\overline{p})}{n}}\,\rho^{\frac{n\,+\,\overline{p}}{n}}}\,\,Y_{j}^{1\,+\,\frac{\overline{p}}{n}}\,.
\end{align} Putting together (\ref{eq:Yj+1}), (\ref{eq:measure}) and \eqref{eq:clef02},
and recalling that $q:=\,\overline{p}\left(1+\frac{2}{n}\right)$,
we end up with\begin{align}\label{eq:dis_chiave}
Y_{j+1}\,&\le\,\frac{C_{2}\,4^{j\,\frac{P^{2}\,(n\,+\,\overline{p})}{nq}}}{(1-\sigma)^{\frac{\mathcal{P}\,(n\,+\,\overline{p})\,P}{nq}}\,\rho^{\frac{P\,(n\,+\,\overline{p})}{nq}}}\,\,\frac{2^{(j+1)P\left(1\,-\,\frac{P}{q}\right)}}{k^{P\,-\,\frac{P^{2}}{q}}}\,\,Y_{j}^{1\,+\,\frac{P\,\overline{p}}{nq}}\nonumber\\
&\leq\,\frac{C_{3}\,4^{jP\left[\frac{P\,(n\,+\,\overline{p})}{nq}\,+\,1\,-\,\frac{P}{q}\right]}}{(1-\sigma)^{\frac{\mathcal{P}\,(n\,+\,\overline{p})\,P}{nq}}\,\rho^{\frac{P\,(n\,+\,\overline{p})}{nq}}}\,\,k^{\frac{P^{2}}{q}\,-\,P}\,\,Y_{j}^{1\,+\,\frac{P\,\overline{p}}{nq}}\nonumber\\
&=\,\frac{C_{3}\,b^{j}}{\left[(1-\sigma)^{\mathcal{P}}\rho\right]^{\frac{P\,(n\,+\,\overline{p})}{\overline{p}\,(n\,+\,2)}}}\,\,\mathcal{A}_{k}\,\,Y_{j}^{1\,+\,\frac{P}{n\,+\,2}}\,,
\end{align} where $C_{2}$ and $C_{3}$ are positive constants depending only
$n$, $\Lambda$, $\mathbf{p}$ and $\delta$, 
\[
b:=\,4^{P\,+\,\frac{P^{2}}{n\,+\,2}}\,,
\]
\[
\mathcal{A}_{k}:=\,k^{\frac{P^{2}n}{\overline{p}\,(n\,+\,2)}\,-\,P}.
\]
We are thus in a position to apply the result in Lemma \ref{lem:Giusti}
with the choices 
\[
C\,=\,C_{3}\,\mathcal{A}_{k}\left[(1-\sigma)^{\mathcal{P}}\rho\right]^{-\,\frac{P\,(n\,+\,\overline{p})}{\overline{p}\,(n\,+\,2)}}\,,\,\,\,\,\,\,\,\,\,\,\,\,\mu\,=\,\frac{P}{n+2}\,,
\]
where $C_{3}$ is the constant from \eqref{eq:dis_chiave}, provided
that 
\begin{equation}
Y_{0}\,\leq\,C^{-\,\frac{1}{\mu}}\,b^{-\,\frac{1}{\mu^{2}}}\,=\,C_{3}\left[(1-\sigma)^{\mathcal{P}}\rho\right]^{\frac{n\,+\,\overline{p}}{\overline{p}}}k^{n\,+\,2\,-\,\frac{nP}{\overline{p}}}\,.\label{eq:Y0}
\end{equation}
Since 
\[
Y_{0}:=\iint_{Q_{\rho}}u_{+}^{P}\,dx\,dt\,\,\,\,\,\,\,\,\,\,\,\,\,\,\,\,\mathrm{and}\,\,\,\,\,\,\,\,\,\,\,\,\,\,\,\,\overline{p}\,(n+2)-nP\,=\,(q-P)\,n\,>\,0\,,
\]
and since $\vert Q_{\rho}\vert=2^{n}\,\rho^{\frac{n\,+\,\overline{p}}{\overline{p}}}$,
we see that (\ref{eq:Y0}) is satisfied provided that 
\[
k\,\ge\,C_{4}\left[(1-\sigma)^{-\,\frac{\mathcal{P}\,(n\,+\,\overline{p})}{\overline{p}}}\tiltfiint_{Q_{\rho}}u_{+}^{P}\,dx\,dt\right]^{\frac{\overline{p}}{\overline{p}\,(n\,+\,2)\,-\,nP}}=:\mathcal{B}\,,
\]
for a positive constant $C_{4}$ depending only $n$, $\Lambda$,
$\mathbf{p}$ and $\max\,\{\delta_{1},\ldots,\delta_{n}\}$. Therefore,
choosing $k=\max\,\{1,\rho^{1/P},\mathcal{B}\}$, by Lemma \ref{lem:Giusti}
we have
\[
\iint_{Q_{\sigma\rho}}(u-k)_{+}^{P}\,dx\,dt\,\le\,Y_{j}\rightarrow0\,\,\,\,\,\,\,\,\mathrm{as}\,\,j\to\infty\,.
\]
This implies that the integrand above is zero almost everywhere in
$Q_{\sigma\rho}$, and hence 
\[
\underset{Q_{\sigma\rho}}{\mathrm{ess}\,\sup}\,\,u\,\leq\,\max\left\{ 1,\,\rho^{1/P},\,C_{4}\left[(1-\sigma)^{-\,\frac{\mathcal{P}\,(n\,+\,\overline{p})}{\overline{p}}}\tiltfiint_{Q_{\rho}}u_{+}^{P}\,dx\,dt\right]^{\frac{\overline{p}}{\overline{p}\,(n\,+\,2)\,-\,nP}}\right\} .
\]
We have thus proved estimate (\ref{eq:supremum}).\\
$\hspace*{1em}$Finally, if $u\in\mathcal{DG}^{-}(\mathbf{p},\{\delta_{i}\},\Omega_{T},\mathcal{C})$,
the previous arguments, with minor modifications, yield the lower
bound (\ref{eq:infimum}). We leave the details to the reader.\end{proof}

\noindent $\hspace*{1em}$As an easy consequence of Definition \ref{def:DeGiorgi},
Remark \ref{def:ossDG} and Theorem \ref{thm:main}, we obtain the
following local boundedness result.
\begin{cor}
\noindent \label{cor:corol1}Let $\mathbf{p}=(p_{1},\ldots,p_{n})$
satisfy $(\ref{eq:range})$. Assume that $(\ref{eq:supercritical})$
holds and that $u\in\mathcal{DG}(\mathbf{p},\{\delta_{i}\},\Omega_{T},\mathcal{C})$.
Then $u$ is locally bounded and, for every cylinder $Q_{\rho}(x_{0},t_{0})\Subset\Omega_{T}$
and every $\sigma\in(0,1)$, we have
\begin{equation}
\underset{Q_{\sigma\rho}(x_{0},t_{0})}{\mathrm{ess}\,\sup}\,\vert u\vert\,\leq\,\max\left\{ 1,\,\rho^{1/P},\,C\left[(1-\sigma)^{-\,\frac{\mathcal{P}\,(n\,+\,\overline{p})}{\overline{p}}}\tiltfiint_{Q_{\rho}(x_{0},t_{0})}\vert u\vert^{P}\,dx\,dt\right]^{\frac{\overline{p}}{\overline{p}\,(n\,+\,2)\,-\,nP}}\right\} ,\label{eq:sup2}
\end{equation}
where $\mathcal{P}:=\max\,\{p_{1},\ldots,p_{n}\}$, $P:=\max\,\{2,\mathcal{P}\}$
and $C$ is a positive constant depending only on $n$, $\Lambda$,
$\mathbf{p}$ and $\max\,\{\delta_{1},\ldots,\delta_{n}\}$. In particular,
if $u$ is a local weak solution of $(\ref{eq:equation})$ under assumptions
$(\ref{eq:coeff})$, $(\ref{eq:range})$ and $(\ref{eq:supercritical})$,
then $u$ belongs to $L_{loc}^{\infty}(\Omega_{T})$ and satisfies
the estimate $(\ref{eq:sup2})$.
\end{cor}

\noindent $\hspace*{1em}$By iterating the estimate of the previous
corollary, we can actually get a better bound for the essential supremum.
We first note that, by the upper bound for the exponents $p_{i}$
in (\ref{eq:range}) and inequality (\ref{eq:supercritical}), we
have 
\[
P\,>\,P\,-\,\frac{n}{\overline{p}}\left[\overline{p}\left(1+\frac{2}{n}\right)-P\right]=:\,\omega(n,\mathbf{p})\,.
\]
Moreover, from the definitions of $\omega(n,\mathbf{p})$ and $P$,
we see that 
\[
\omega(n,\mathbf{p})\,=\,P-2+n\left(\frac{P}{\overline{p}}\,-1\right)\,\geq\,0\,.
\]
Now we are in a position to prove the following result.
\begin{thm}
\label{thm:main2-1}Under the assumptions of Corollary \ref{cor:corol1},
for every cylinder $Q_{2\rho}(x_{0},t_{0})\Subset\Omega_{T}$ and
every $\nu\in(\omega(n,\mathbf{p}),P]$ we have
\begin{equation}
\underset{Q_{\rho}(x_{0},t_{0})}{\mathrm{ess}\,\sup}\,\vert u\vert\,\leq\,C\left[\left(\tiltfiint_{Q_{2\rho}(x_{0},t_{0})}\vert u\vert^{\nu}\,dx\,dt\right)^{\frac{1}{\nu\,-\,\omega(n,\mathbf{p})}}+1+\rho^{1/P}\right],\label{eq:BOUND2-1}
\end{equation}
for a positive constant $C$ depending only on $n$, $\Lambda$, $\mathbf{p}$,
$\nu$ and $\max\,\{\delta_{1},\ldots,\delta_{n}\}$.
\end{thm}

\noindent \begin{proof}[\bfseries{Proof}]Let \textit{$\nu\in(\omega(n,\mathbf{p}),P]$}.
Observe that, in the case $\nu=P$, estimate (\ref{eq:BOUND2-1})
is a straightforward consequence of Corollary \ref{cor:corol1}. Indeed,
by choosing $\sigma=\frac{1}{2}$ and replacing $\rho$ with $2\rho$
in (\ref{eq:sup2}), one immediately obtains the bound (\ref{eq:BOUND2-1}).
Therefore, from now on we assume that $\nu<P$. Let $Q_{2\rho}(x_{0},t_{0})$
be compactly contained in $\Omega_{T}$ and define 
\[
\rho_{j}:=\,2\rho\,-\,\frac{\rho}{2^{j}}\,,\,\,\,\,\,\,\,\,\,\,\,\,\,\,\mathcal{Q}_{j}:=\,Q_{\rho_{j}}(x_{0},t_{0})\,,\,\,\,\,\,\,\,\,\,\,j\in\mathbb{N}_{0}\,.
\]
Using (\ref{eq:sup2}) with $\sigma=\frac{\rho_{j}}{\rho_{j+1}}$
and noting that $\rho_{j+1}<2\rho$, we obtain\begin{align*}
M_{j}:=\,\underset{\mathcal{Q}_{j}}{\mathrm{ess}\,\sup}\,\vert u\vert\,&\leq\,C\left(2^{j\,\frac{\mathcal{P}\,(n\,+\,\overline{p})}{\overline{p}}}\tiltfiint_{\mathcal{Q}_{j+1}}\vert u\vert^{P}\,dx\,dt\right)^{\frac{\overline{p}}{\overline{p}\,(n\,+\,2)\,-\,nP}}+\,1\,+\,\rho_{j+1}^{1/P}\\
&\leq\,C\,M_{j+1}^{\frac{\overline{p}\,(P-\nu)}{\overline{p}\,(n\,+\,2)\,-\,nP}}\left(2^{j\,\frac{\mathcal{P}\,(n\,+\,\overline{p})}{\overline{p}}}\tiltfiint_{\mathcal{Q}_{j+1}}\vert u\vert^{\nu}\,dx\,dt\right)^{\frac{\overline{p}}{\overline{p}\,(n\,+\,2)\,-\,nP}}+\,1\,+\,(2\rho)^{1/P}\,,
\end{align*}where $C$ is a positive constant depending only on $n$, $\Lambda$,
$\mathbf{p}$ and $\max\,\{\delta_{1},\ldots,\delta_{n}\}$. Since
$\omega(n,\mathbf{p})<\nu<P$, we see that the exponent of $M_{j+1}$
in the last expression belongs to the interval $(0,1)$. Thus we can
use Young's inequality with $\varepsilon>0$ to obtain
\[
M_{j}\,\leq\,\varepsilon\,M_{j+1}\,+\,c(\varepsilon)\,b^{j}\left(\tiltfiint_{Q_{2\rho}(x_{0},t_{0})}\vert u\vert^{\nu}\,dx\,dt\right)^{\frac{1}{\nu\,-\,\omega(n,\mathbf{p})}}+\,1\,+\,(2\rho)^{1/P}\,,
\]
where we have also used the fact that $\mathcal{Q}_{j+1}$ is contained
in $Q_{2\rho}(x_{0},t_{0})$. Here we emphasize that the positive
constant $c(\varepsilon)$ depends on $\varepsilon$, in addition
to $n$, $\Lambda$, $\mathbf{p}$, $\nu$ and $\max\,\{\delta_{1},\ldots,\delta_{n}\}$.
The constant $b>1$ depends on $n$, $\mathbf{p}$ and $\nu$. By
iterating the last estimate, we end up with 
\begin{equation}
M_{0}\,\leq\,\varepsilon^{N}M_{N}\,+\,c(\varepsilon)\left(\sum_{\ell\,=\,0}^{N-1}(\varepsilon b)^{\ell}\right)\left(\tiltfiint_{Q_{2\rho}(x_{0},t_{0})}\vert u\vert^{\nu}\,dx\,dt\right)^{\frac{1}{\nu\,-\,\omega(n,\mathbf{p})}}+\,[1+(2\rho)^{1/P}]\,\sum_{\ell\,=\,0}^{N-1}\varepsilon^{\ell}\,,\label{eq:itera1}
\end{equation}
for any $N\in\mathbb{N}$. Choose $\varepsilon$ so small that $\varepsilon b<1$.
Then the sums in the previous inequality converge as $N\rightarrow\infty$.
Since the sequence $\{M_{N}\}$ is bounded due to the local boundedness
of $u$, the first term on the right-hand side of (\ref{eq:itera1})
tends to zero in the limit as $N\rightarrow\infty$, which yields
the bound (\ref{eq:BOUND2-1}).\end{proof}

\subsection{The case $\overline{p}\protect\leq\frac{2n}{n+2}$ }

\noindent $\hspace*{1em}$We now turn our attention to the range 
\begin{equation}
\overline{p}\,\leq\,\frac{2n}{n+2}\,,\label{eq:sub-crit}
\end{equation}
and recall that we also require (\ref{eq:extra_integr}) in this case.
In the previous section, the condition $u\in L_{loc}^{P}(\Omega_{T})$,
combined with (\ref{eq:supercritical}), implies (\ref{eq:extra_integr}),
so this assumption was not needed. The argument consists of two parts.
First, we establish the local boundedness from above (resp. from below)
for functions \textit{$u\in\mathcal{DG}^{+}(\mathbf{p},\{\delta_{i}\},\Omega_{T},\mathcal{C})$}
(resp. \textit{$u\in\mathcal{DG}^{-}(\mathbf{p},\{\delta_{i}\},\Omega_{T},\mathcal{C})$})
without an explicit bound. Then, we follow the approach of DiBenedetto
in \cite[Chapter V, Section 10]{DiBe} to derive an explicit estimate
for the essential supremum of $u_{+}$ and $u_{-}$, respectively. 
\begin{thm}
\noindent \label{thm:main-subcrit}Let $u\in\mathcal{DG}^{+}(\mathbf{p},\{\delta_{i}\},\Omega_{T},\mathcal{C})$,
where $\mathbf{p}=(p_{1},\ldots,p_{n})$ satisfies $(\ref{eq:range})$.
Suppose that $(\ref{eq:sub-crit})$ holds and that $u$ satisfies
the extra integrability condition $(\ref{eq:extra_integr})$. Then
$u$ is locally bounded from above and, if $Q_{2\rho}(x_{0},t_{0})$
is compactly contained in $\Omega_{T}$, we have the following explicit
bound: 
\begin{equation}
\underset{Q_{\rho}(x_{0},t_{0})}{\mathrm{ess}\,\sup}\,\,u_{+}\,\leq\,C\left[\left(\tiltfiint_{Q_{2\rho}(x_{0},t_{0})}u_{+}^{m}\,dx\,dt\right)^{\frac{1}{m\,-\,\frac{n}{\overline{p}}\,(2\,-\,\overline{p})}}+1+\sqrt{\rho}\,\right],\label{eq:BOUND3-1}
\end{equation}
where $C$ is a positive constant depending only on $n$, $\Lambda$,
$\mathbf{p}$, $m$ and $\max\,\{\delta_{1},\ldots,\delta_{n}\}$.
Similarly, if $u\in\mathcal{DG}^{-}(\mathbf{p},\{\delta_{i}\},\Omega_{T},\mathcal{C})$
under assumptions $(\ref{eq:range})$, $(\ref{eq:extra_integr})$
and $(\ref{eq:sub-crit})$, then $u$ is locally bounded from below
and, if $Q_{2\rho}(x_{0},t_{0})\Subset\Omega_{T}$, we have
\begin{equation}
\underset{Q_{\rho}(x_{0},t_{0})}{\mathrm{ess}\,\sup}\,\,u_{-}\,\leq\,C\left[\left(\tiltfiint_{Q_{2\rho}(x_{0},t_{0})}u_{-}^{m}\,dx\,dt\right)^{\frac{1}{m\,-\,\frac{n}{\overline{p}}\,(2\,-\,\overline{p})}}+1+\sqrt{\rho}\,\right].\label{eq:BOUND3-2}
\end{equation}
\end{thm}

\noindent \begin{brem}\label{thm:Rk2}In view of Remark \ref{def:ossDG}
and Theorem \ref{thm:main-subcrit}, under assumptions (\ref{eq:coeff}),
(\ref{eq:range}), (\ref{eq:extra_integr}) and (\ref{eq:sub-crit}),
local weak subsolutions $u$ to (\ref{eq:equation}) are locally bounded
from above in $\Omega_{T}$, whereas local weak supersolutions are
locally bounded from below.\end{brem}\vspace{-2mm}

\noindent \begin{proof}[\bfseries{Proof of Theorem~\ref{thm:main-subcrit}}]After
a translation, we may assume again that $(x_{0},t_{0})=(0,0)$. By
(\ref{eq:extra_integr}) and (\ref{eq:sub-crit}) we have 
\begin{equation}
m\,>\,2n\left(\frac{1}{\overline{p}}\,-\,\frac{1}{2}\right)\geq\,2\,.\label{eq:m-bound}
\end{equation}
For the sake of readability, we divide the proof into two steps. \medskip{}

\noindent \textbf{Step 1: local boundedness.} Let us first assume
that \textit{$u\in\mathcal{DG}^{+}(\mathbf{p},\{\delta_{i}\},\Omega_{T},\mathcal{C})$.}
We choose $\sigma$, $\delta$, the cylinders $\mathcal{Q}_{j}$,
and the functions $\tilde{\zeta}_{i,j}$ and $\psi_{j}$ as in the
proof of Theorem \ref{thm:main}, and we show that $u$ is essentially
bounded from above in $Q_{\sigma\rho}$. To this end, we define 
\[
\mathfrak{q}:=\,\frac{m-2\,\overline{p}\left(\frac{1}{n}+\frac{1}{2}\right)}{m-2}\,\geq\,1\,,\,\,\,\,\,\,\,\,\,\,\,\,\,\,\,\vartheta:=\,\frac{\overline{p}\left(\frac{1}{n}+\frac{1}{2}\right)}{\mathfrak{q}}\,\in(0,1]\,.
\]
The lower bound for $\mathfrak{q}$ follows from (\ref{eq:sub-crit}),
and from this we deduce the range of $\vartheta$. Moreover, we set
\[
Y_{j}:=\iint_{\mathcal{Q}_{j}}(u-k_{j})_{+}^{2}\,dx\,dt\,,\,\,\,\,\,\,\,\,\,\,\,\,\,\,k_{j}:=\,2k\,-\,\frac{k}{2^{j}}\,,\,\,\,\,\,\,\,\,\,\,j\in\mathbb{N}_{0}\,.
\]
We require $k\geq1$ as before. If strict inequality holds in (\ref{eq:sub-crit}),
then $\mathfrak{q}>1$ and we have by Hölder's inequality\begin{align}\label{eq:sub_1}
Y_{j+1}\,&=\iint_{\mathcal{Q}_{j+1}}(u-k_{j+1})_{+}^{2\vartheta}\,(u-k_{j+1})_{+}^{2(1-\vartheta)}\,dx\,dt\nonumber\\
&\le\left(\iint_{\mathcal{Q}_{j+1}}(u-k_{j+1})_{+}^{2\vartheta \mathfrak{q}}\,dx\,dt\right)^{\frac{1}{\mathfrak{q}}}\left(\iint_{\mathcal{Q}_{j+1}}(u-k_{j+1})_{+}^{2(1-\vartheta)\,\mathfrak{q}'}\,dx\,dt\right)^{\frac{\mathfrak{q}-1}{\mathfrak{q}}}\nonumber\\
&=\left(\iint_{\mathcal{Q}_{j+1}}(u-k_{j+1})_{+}^{2\,\overline{p}\left(\frac{1}{n}+\frac{1}{2}\right)}\,dx\,dt\right)^{\frac{1}{\mathfrak{q}}}\left(\iint_{\mathcal{Q}_{j+1}}(u-k_{j+1})_{+}^{m}\,dx\,dt\right)^{\frac{\mathfrak{q}-1}{\mathfrak{q}}}\nonumber\\
&\leq\left(\iint_{\mathcal{Q}_{j+1}}(u-k_{j+1})_{+}^{2\,\overline{p}\left(\frac{1}{n}+\frac{1}{2}\right)}\,dx\,dt\right)^{\frac{1}{\mathfrak{q}}}\left(\iint_{\mathcal{Q}_{0}}u_{+}^{m}\,dx\,dt\right)^{\frac{\mathfrak{q}-1}{\mathfrak{q}}}.
\end{align}If, instead, (\ref{eq:sub-crit}) holds with equality, then $\mathfrak{q}=1$
and the previous estimate is valid without applying Hölder's inequality,
provided that we adopt the standard convention $0^{0}=1$ if 
\[
I:=\iint_{\mathcal{Q}_{0}}u_{+}^{m}\,dx\,dt\,=\,0\,.
\]
We note that inequality \eqref{eq:big_bis} is valid for the first
integral in the last line of \eqref{eq:sub_1}, despite the different
definition of $k_{j}$. The upper bound for the exponents $p_{i}$
in (\ref{eq:range}), combined with (\ref{eq:sub-crit}), implies
that $p_{i}<2$ for every $i\in\{1,\ldots,n\}$; hence $P=2$. Therefore,
arguing as in \eqref{eq:big_bis002}, but this time using that $k_{j}<2k$,
we have 
\[
\sum_{i=1}^{n}\iint_{\mathcal{Q}_{j}}(u-k_{j+1})_{+}^{p_{i}}\,dx\,dt\,\leq\,2^{\mathcal{P}+1}\,n\left[Y_{j}\,+\,(1+k^{2})\,\vert A_{j+1}\vert\right],
\]
where $A_{j+1}$ is defined as in (\ref{eq:superlevel}). Combining
the above estimate with \eqref{eq:big_bis} and recalling that $k\geq1$,
we obtain
\[
\iint_{\mathcal{Q}_{j+1}}(u-k_{j+1})_{+}^{2\,\overline{p}\left(\frac{1}{n}+\frac{1}{2}\right)}\,dx\,dt\,\leq\,\frac{C_{1}\,2^{j\,\frac{\mathcal{P}\,(n\,+\,\overline{p})}{n}}}{(1-\sigma)^{\frac{\mathcal{P}\,(n\,+\,\overline{p})}{n}}}\left[\frac{Y_{j}}{\rho}\,+\left(1+\,\frac{2k^{2}}{\rho}\right)\vert A_{j+1}\vert\,\right]^{\frac{n\,+\,\overline{p}}{n}},
\]
where $C_{1}$ is a positive constant depending only on $n$, $\Lambda$,
$\mathbf{p}$ and $\delta$. Without loss of generality, we can now
assume that $k\geq\max\,\{1,\sqrt{\rho}\}$. It then follows from
the previous inequality that\\
\[
\iint_{\mathcal{Q}_{j+1}}(u-k_{j+1})_{+}^{2\,\overline{p}\left(\frac{1}{n}+\frac{1}{2}\right)}\,dx\,dt\,\leq\,\frac{C_{1}\,2^{j\,\frac{\mathcal{P}\,(n\,+\,\overline{p})}{n}}}{\left[(1-\sigma)^{\mathcal{P}}\rho\right]^{\frac{n\,+\,\overline{p}}{n}}}\left[Y_{j}\,+\,k^{2}\,\vert A_{j+1}\vert\,\right]^{\frac{n\,+\,\overline{p}}{n}},
\]
and reasoning as in (\ref{eq:measure}) we have 
\[
\vert A_{j+1}\vert\,\leq\,\frac{4^{j+1}}{k^{2}}\,Y_{j}\,.
\]
At this point, noting that in the current parameter range one has
$\mathcal{P}<2$, we can combine the two previous estimates with \eqref{eq:sub_1}
to conclude that
\begin{equation}
Y_{j+1}\,\leq\,\frac{C_{2}\,b^{j}}{\left[(1-\sigma)^{\mathcal{P}}\rho\right]^{\frac{n\,+\,\overline{p}}{n\mathfrak{q}}}}\,\,I^{\frac{\mathfrak{q}-1}{\mathfrak{q}}}\,\,Y_{j}^{\frac{n\,+\,\overline{p}}{n\mathfrak{q}}}\,,\label{eq:recursive}
\end{equation}
where $I$ denotes the integral over $\mathcal{Q}_{0}$ appearing
in the last line of \eqref{eq:sub_1}, $b>1$ is a constant depending
only on $n$, $\mathbf{p}$ and $m$, while $C_{2}$ is a positive
constant depending only on $n$, $\Lambda$, $\mathbf{p}$, $\delta$
and $m$. Now we consider the following four cases:\\
\[
(i)\,\,\,\,\,\,I=0\,\,\,\,\mathrm{and}\,\,\,\,\mathfrak{q}>1\,;
\]
\[
(ii)\,\,\,\,I=0\,\,\,\,\mathrm{and}\,\,\,\,\mathfrak{q}=1\,;
\]
\[
(iii)\,\,\,I>0\,\,\,\,\mathrm{and}\,\,\,\,\mathfrak{q}=1\,;
\]
\[
(iv)\,\,\,I>0\,\,\,\,\mathrm{and}\,\,\,\,\mathfrak{q}>1\,.
\]
$\hspace*{1em}$In case ($i$), from the recursive inequality (\ref{eq:recursive})
we deduce 
\[
\iint_{Q_{\sigma\rho}}(u-2k)_{+}^{2}\,dx\,dt\,\le\,Y_{j+1}\,=\,0
\]
for every $j\in\mathbb{N}_{0}$, which implies that $u\leq2k\,$ a.e.
in $Q_{\sigma\rho}$.\\
$\hspace*{1em}$In cases ($ii$) and ($iii$), inequality (\ref{eq:recursive})
reduces to
\[
Y_{j+1}\,\leq\,\frac{C_{2}\,b^{j}}{\left[(1-\sigma)^{\mathcal{P}}\rho\right]^{\frac{n\,+\,\overline{p}}{n}}}\,\,Y_{j}^{1\,+\,\frac{\overline{p}}{n}}\,.
\]
Thus, the sequence $\{Y_{j}\}$ satisfies the recursive estimate of
Lemma \ref{lem:Giusti} with 
\[
C\,=\,C_{2}\left[(1-\sigma)^{\mathcal{P}}\rho\right]^{-\,\frac{n\,+\,\overline{p}}{n}},\,\,\,\,\,\,\,\,\,\,\,\,\mu\,=\,\frac{\overline{p}}{n}\,.
\]
By Lemma \ref{lem:Giusti}, the sequence $\{Y_{j}\}$ converges to
zero provided that
\[
\iint_{Q_{\rho}}(u-k)_{+}^{2}\,dx\,dt\,=\,Y_{0}\,\leq\,C^{-\,\frac{1}{\mu}}\,b^{-\,\frac{1}{\mu^{2}}}\,=\,C_{2}\left[(1-\sigma)^{\mathcal{P}}\rho\right]^{\frac{n\,+\,\overline{p}}{\overline{p}}}.
\]
The integral on the left-hand side vanishes in the limit as $k\rightarrow\infty$
by the dominated convergence theorem and the expression on the right-hand
side is independent of $k$, so the estimate must indeed hold for
large $k\geq\max\,\{1,\sqrt{\rho}\}$. Thus, we have 
\[
\iint_{Q_{\sigma\rho}}(u-2k)_{+}^{2}\,dx\,dt\,\le\,Y_{j}\rightarrow0\,\,\,\,\,\,\,\,\mathrm{as}\,\,j\to\infty\,,
\]
which implies that $u\leq2k\,$ a.e. in $Q_{\sigma\rho}$ for some
sufficiently large $k$.\\
$\hspace*{1em}$It remains to analyze case ($iv$). In this case,
(\ref{eq:sub-crit}) holds with strict inequality and, by the definition
of $\mathfrak{q}$ and the lower bound for $m$ in (\ref{eq:m-bound}),
we see that\begin{align*}
\frac{n+\overline{p}}{n\mathfrak{q}}\,&=\left(\frac{n+\overline{p}}{n}\right)\frac{m-2}{m-2\,\overline{p}\left(\frac{1}{n}+\frac{1}{2}\right)}\,=\left(\frac{n+\overline{p}}{n}\right)\left(1+\,\frac{2\,\overline{p}\left(\frac{1}{n}+\frac{1}{2}\right)-2}{m-2\,\overline{p}\left(\frac{1}{n}+\frac{1}{2}\right)}\right)\\
&>\left(\frac{n+\overline{p}}{n}\right)\left(1+\,\frac{\overline{p}\left(\frac{1}{n}+\frac{1}{2}\right)-1}{n\left(\frac{1}{\overline{p}}-\frac{1}{2}\right)-\overline{p}\left(\frac{1}{n}+\frac{1}{2}\right)}\right)=\,1\,.
\end{align*}Therefore, the sequence $\{Y_{j}\}$ satisfies the recursive estimate
of Lemma \ref{lem:Giusti} with
\[
C\,=\,C_{2}\left[(1-\sigma)^{\mathcal{P}}\rho\right]^{-\,\frac{n\,+\,\overline{p}}{n\mathfrak{q}}}\,I^{\frac{\mathfrak{q}-1}{\mathfrak{q}}},\,\,\,\,\,\,\,\,\,\,\,\,\mu\,=\,\frac{n+\overline{p}}{n\mathfrak{q}}\,-1\,>\,0\,.
\]
By Lemma \ref{lem:Giusti}, the sequence $\{Y_{j}\}$ converges to
zero provided that
\[
\iint_{Q_{\rho}}(u-k)_{+}^{2}\,dx\,dt\,=\,Y_{0}\,\leq\,C^{-\,\frac{1}{\mu}}\,b^{-\,\frac{1}{\mu^{2}}}\,=\,C_{2}\left[(1-\sigma)^{\mathcal{P}}\rho\right]^{\frac{n\,+\,\overline{p}}{n\,+\,\overline{p}\,-\,n\mathfrak{q}}}\,I^{-\,\frac{n(\mathfrak{q}-1)}{n\,+\,\overline{p}\,-\,n\mathfrak{q}}}\,.
\]
Once more, the integral on the left-hand side vanishes in the limit
as $k\rightarrow\infty$ and the expression on the right-hand side
is independent of $k$; hence the estimate must hold for large $k\geq\max\,\{1,\sqrt{\rho}\}$.
Reasoning as before, we conclude again that $u\leq2k\,$ a.e. in $Q_{\sigma\rho}$
for some sufficiently large $k$.\\
\textcolor{red}{$\hspace*{1em}$}If \textit{$u\in\mathcal{DG}^{-}(\mathbf{p},\{\delta_{i}\},\Omega_{T},\mathcal{C})$},
the same arguments, with minor modifications, give $u\geq-2k\,$ a.e.
in $Q_{\sigma\rho}$ for some $k$ large enough. We leave the details
to the reader.\\
\\
\textbf{Step 2: explicit $L^{\infty}$ bounds.} Let us again assume
that\textcolor{red}{{} }\textit{$u\in\mathcal{DG}^{+}(\mathbf{p},\{\delta_{i}\},\Omega_{T},\mathcal{C})$}.
Knowing that $u$ is locally essentially bounded from above, we now
proceed to establish the explicit bound in (\ref{eq:BOUND3-1}). By
combining (\ref{eq:sub-crit}) and (\ref{eq:m-bound}), we can introduce
a positive parameter as follows: 
\[
\nu\,:=\,m-2\,\overline{p}\left(\frac{1}{n}+\frac{1}{2}\right)>\,0\,.
\]
Let $k_{j}$ be defined again as in (\ref{eq:levels}), and set 
\[
Z_{j}:=\iint_{\mathcal{Q}_{j}}(u-k_{j})_{+}^{m}\,dx\,dt\,,\,\,\,\,\,\,\,\,\,\,j\in\mathbb{N}_{0}\,.
\]
Then we have 
\[
Z_{j+1}\,\leq\,\Vert u_{+}\Vert_{L^{\infty}(\mathcal{Q}_{0})}^{\nu}\iint_{\mathcal{Q}_{j+1}}(u-k_{j+1})_{+}^{2\,\overline{p}\left(\frac{1}{n}+\frac{1}{2}\right)}\,dx\,dt\,.
\]
Using \eqref{eq:big_bis} to estimate the last integral, we conclude
that\begin{align*}
Z_{j+1}\,\leq \,&\, C_{1}\,\frac{\Vert u_{+}\Vert_{L^{\infty}(\mathcal{Q}_{0})}^{\nu}\,2^{j\,\frac{\mathcal{P}\,(n\,+\,\overline{p})}{n}}}{(1-\sigma)^{\frac{\mathcal{P}\,(n\,+\,\overline{p})}{n}}}\nonumber\\
&\times\left[\frac{1}{\rho}\iint_{\mathcal{Q}_{j}}(u-k_{j+1})_{+}^{2}\,dx\,dt\,+\,\frac{1}{\rho}\,\sum_{i=1}^{n}\iint_{\mathcal{Q}_{j}}(u-k_{j+1})_{+}^{p_{i}}\,dx\,dt\,+\,\vert A_{j+1}\vert\,\right]^{\frac{n\,+\,\overline{p}}{n}}.
\end{align*} Exploiting the definitions of $A_{j+1}$ and $k_{j}$, applying Young's
inequality with exponents $\frac{2}{p_{i}}$, using the fact that
$k\geq1$ and $m>2$, and arguing as in (\ref{eq:measure}), we end
up with\begin{align*}
&\iint_{\mathcal{Q}_{j}}(u-k_{j+1})_{+}^{2}\,dx\,dt\,+\,\sum_{i=1}^{n}\iint_{\mathcal{Q}_{j}}(u-k_{j+1})_{+}^{p_{i}}\,dx\,dt\\
&\,\,\,\,\,\,\,\leq\iint_{A_{j+1}}(u-k_{j})_{+}^{2}\,dx\,dt\,+\,\sum_{i=1}^{n}\iint_{A_{j+1}}(u-k_{j})_{+}^{p_{i}}\,dx\,dt\\
&\,\,\,\,\,\,\,\leq\,(n+1)\left[\iint_{A_{j+1}}(u-k_{j})_{+}^{2}\,dx\,dt\,+\,\vert A_{j+1}\vert\right]\\
&\,\,\,\,\,\,\,\leq\,(n+1)\left(1+\,\frac{4^{j+1}}{k^{2}}\right)\iint_{A_{j+1}}\frac{(u-k_{j})_{+}^{m}}{(u-k_{j})_{+}^{m-2}}\,dx\,dt\\
&\,\,\,\,\,\,\,\leq\,(n+1)\left(1+4^{j+1}\right)\iint_{A_{j+1}}\frac{(u-k_{j})_{+}^{m}}{(k_{j+1}-k_{j})^{m-2}}\,dx\,dt\\
&\,\,\,\,\,\,\,\leq\,2\,(n+1)\,4^{j+1}\,\frac{2^{(j+1)(m-2)}}{k^{m-2}}\,Z_{j}\,=\,2^{m+1}(n+1)\,\frac{2^{mj}}{k^{m-2}}\,Z_{j}\,.
\end{align*}Combining the two previous estimates and recalling that $\mathcal{Q}_{0}=Q_{\rho}$,
we obtain
\begin{equation}
Z_{j+1}\,\leq\,C_{2}\,\frac{\Vert u_{+}\Vert_{L^{\infty}(Q_{\rho})}^{\nu}\,2^{j\,\frac{\mathcal{P}\,(n\,+\,\overline{p})}{n}}}{(1-\sigma)^{\frac{\mathcal{P}\,(n\,+\,\overline{p})}{n}}}\,\left[\frac{2^{mj}\,Z_{j}}{\rho\,k^{m-2}}\,+\,\vert A_{j+1}\vert\,\right]^{\frac{n\,+\,\overline{p}}{n}},\label{eq:sub_02}
\end{equation}
where $C_{2}=C_{2}(n,\Lambda,\mathbf{p},\delta,m)>0$. We now use
again the definitions of $A_{j+1}$ and $k_{j}$, together with $k\geq\max\,\{1,\sqrt{\rho}\}$,
to deduce that\begin{align}\label{eq:sub_03}
\vert A_{j+1}\vert & =\iint_{A_{j+1}}\frac{(u-k_{j})_{+}^{m}}{(u-k_{j})_{+}^{m}}\,dx\,dt\,\le\iint_{A_{j+1}}\frac{(u-k_{j})_{+}^{m}}{(k_{j+1}-k_{j})^{m}}\,dx\,dt\nonumber\\
&=\,\frac{2^{m(j+1)}}{k^{m}}\iint_{A_{j+1}}(u-k_{j})_{+}^{m}\,dx\,dt\,\leq\,2^{m}\,\frac{2^{mj}\,Z_{j}}{\rho\,k^{m-2}}\,.
\end{align}Joining (\ref{eq:sub_02}) and \eqref{eq:sub_03}, we find 
\[
Z_{j+1}\,\leq\,C_{2}\,\frac{\Vert u_{+}\Vert_{L^{\infty}(Q_{\rho})}^{\nu}\,k^{\frac{n\,+\,\overline{p}}{n}\,(2-m)}\,b^{j}}{\left[(1-\sigma)^{\mathcal{P}}\rho\right]^{\frac{n\,+\,\overline{p}}{n}}}\,\,Z_{j}^{1\,+\,\frac{\overline{p}}{n}}\,,
\]
where $b>1$ is a constant depending only on $n$, $\mathbf{p}$ and
$m$.\\
$\hspace*{1em}$Henceforth, we assume that $\Vert u_{+}\Vert_{L^{\infty}(Q_{\rho})}>0$.
In this case, from Lemma \ref{lem:Giusti} we see that $\{Z_{j}\}$
converges to zero provided that 
\[
\iint_{Q_{\rho}}u_{+}^{m}\,dx\,dt\,=\,Z_{0}\,\leq\,C_{2}\,\Vert u_{+}\Vert_{L^{\infty}(Q_{\rho})}^{-\,\frac{n\,\nu}{\overline{p}}}\,k^{\frac{n\,+\,\overline{p}}{\overline{p}}\,(m-2)}\left[(1-\sigma)^{\mathcal{P}}\rho\right]^{\frac{n\,+\,\overline{p}}{\overline{p}}},
\]
which holds if 
\begin{equation}
k\,\geq\,C_{2}\,\Vert u_{+}\Vert_{L^{\infty}(Q_{\rho})}^{\zeta}\,(1-\sigma)^{-\,\frac{\mathcal{P}}{m-2}}\left(\tiltfiint_{Q_{\rho}}u_{+}^{m}\,dx\,dt\right)^{\frac{\overline{p}}{(n\,+\,\overline{p})(m-2)}},\label{eq:k-bound}
\end{equation}
where\begin{align*}
\zeta\,&=\,\frac{n\nu}{(n+\overline{p})(m-2)}\,=\,\frac{n}{(n+\overline{p})(m-2)}\left[m-2\,\overline{p}\left(\frac{1}{n}+\frac{1}{2}\right)\right]\\
&=\,\frac{n}{n+\overline{p}}\left[1+\,\frac{2-2\,\overline{p}\left(\frac{1}{n}+\frac{1}{2}\right)}{m-2}\right].
\end{align*}If (\ref{eq:sub-crit}) holds with equality, then the numerator in
the last fractional term vanishes and $\zeta=n/(n+\overline{p})<1$.
If, instead, (\ref{eq:sub-crit}) holds with a strict inequality,
then also the last inequality in (\ref{eq:m-bound}) is strict and
we may estimate 
\[
\zeta\,<\,\frac{n}{n+\overline{p}}\left[1+\,\frac{2-2\,\overline{p}\left(\frac{1}{n}+\frac{1}{2}\right)}{2n\left(\frac{1}{\overline{p}}-\frac{1}{2}\right)-2}\right]=\,1\,.
\]
Thus, we have confirmed that in any case $\zeta\in(0,1)$. In addition,
we have shown that if (\ref{eq:k-bound}) and $k\geq\max\,\{1,\sqrt{\rho}\}$
hold, then 
\[
\iint_{Q_{\sigma\rho}}(u-k)_{+}^{m}\,dx\,dt\,\le\,Z_{j}\rightarrow0\,\,\,\,\,\,\,\,\mathrm{as}\,\,j\to\infty\,,
\]
which means that $u\leq k$ almost everywhere in $Q_{\sigma\rho}$.
Therefore, 
\[
\Vert u_{+}\Vert_{L^{\infty}(Q_{\sigma\rho})}\,\leq\,\frac{C_{2}\,\Vert u_{+}\Vert_{L^{\infty}(Q_{\rho})}^{\zeta}}{(1-\sigma)^{\frac{\mathcal{P}}{m-2}}}\left(\tiltfiint_{Q_{\rho}}u_{+}^{m}\,dx\,dt\right)^{\frac{\overline{p}}{(n\,+\,\overline{p})(m-2)}}+1+\sqrt{\rho}\,.
\]
Since $\zeta\in(0,1)$, we may use Young's inequality with $\varepsilon>0$
on the right-hand side to conclude that 
\[
\Vert u_{+}\Vert_{L^{\infty}(Q_{\sigma\rho})}\,\leq\,\varepsilon\,\Vert u_{+}\Vert_{L^{\infty}(Q_{\rho})}\,+\,\frac{c(\varepsilon)}{(1-\sigma)^{\frac{\mathcal{P}}{(m-2)(1-\zeta)}}}\left(\tiltfiint_{Q_{\rho}}u_{+}^{m}\,dx\,dt\right)^{\frac{1}{m\,-\,\frac{n}{\overline{p}}\,(2\,-\,\overline{p})}}+1+\sqrt{\rho}\,,
\]
where we may choose $\varepsilon$ arbitrarily small. Here we emphasize
that the positive constant $c(\varepsilon)$ depends on $\varepsilon$,
in addition to $n$, $\Lambda$, $\mathbf{p}$, $m$ and $\max\,\{\delta_{1},\ldots,\delta_{n}\}$.
Now an iteration similar to that performed in the proof of Theorem
\ref{thm:main2-1} leads to the following estimate:
\begin{equation}
\underset{Q_{\rho}}{\mathrm{ess}\,\sup}\,\,u_{+}\,\leq\,C\left[\left(\tiltfiint_{Q_{2\rho}}u_{+}^{m}\,dx\,dt\right)^{\frac{1}{m\,-\,\frac{n}{\overline{p}}\,(2\,-\,\overline{p})}}+1+\sqrt{\rho}\,\right],\label{eq:stima_u+}
\end{equation}
where $C$ is a positive constant depending only on $n$, $\Lambda$,
$\mathbf{p}$, $m$ and $\max\,\{\delta_{1},\ldots,\delta_{n}\}$.
Note that estimate (\ref{eq:stima_u+}) holds trivially in the case
$\Vert u_{+}\Vert_{L^{\infty}(Q_{\rho})}=0$.\\
$\hspace*{1em}$Finally, if \textit{$u\in\mathcal{DG}^{-}(\mathbf{p},\{\delta_{i}\},\Omega_{T},\mathcal{C})$},
analogous arguments yield the bound (\ref{eq:BOUND3-2}). In this
regard, we leave the straightforward details to the reader.\end{proof}

\noindent $\hspace*{1em}$We conclude this section with the following
corollary on local boundedness, which immediately follows from Definition
\ref{def:DeGiorgi}, Remark \ref{def:ossDG} and Theorem \ref{thm:main-subcrit}.
\begin{cor}
\noindent Let $u\in\mathcal{DG}(\mathbf{p},\{\delta_{i}\},\Omega_{T},\mathcal{C})$,
where $\mathbf{p}=(p_{1},\ldots,p_{n})$ satisfies $(\ref{eq:range})$.
Suppose that $(\ref{eq:sub-crit})$ holds and that $u$ satisfies
the extra integrability condition $(\ref{eq:extra_integr})$. Then
$u$ is locally bounded and, for every cylinder $Q_{2\rho}(x_{0},t_{0})\Subset\Omega_{T}$,
we have 
\begin{equation}
\underset{Q_{\rho}(x_{0},t_{0})}{\mathrm{ess}\,\sup}\,\vert u\vert\,\leq\,C\left[\left(\tiltfiint_{Q_{2\rho}(x_{0},t_{0})}\vert u\vert^{m}\,dx\,dt\right)^{\frac{1}{m\,-\,\frac{n}{\overline{p}}\,(2\,-\,\overline{p})}}+1+\sqrt{\rho}\,\right],\label{eq:estcor}
\end{equation}
where $C$ is a positive constant depending only on $n$, $\Lambda$,
$\mathbf{p}$, $m$ and $\max\,\{\delta_{1},\ldots,\delta_{n}\}$.
In particular, if $u$ is a local weak solution of $(\ref{eq:equation})$
under assumptions $(\ref{eq:coeff})$, $(\ref{eq:range})$, $(\ref{eq:extra_integr})$
and $(\ref{eq:sub-crit})$, then $u$ belongs to $L_{loc}^{\infty}(\Omega_{T})$
and satisfies the estimate $(\ref{eq:estcor})$.
\end{cor}

\section{Critical mass lemma and semicontinuity\label{sec:DeGiorgiLemma}}

\noindent $\hspace*{1em}$In this section, we prove a measure-theoretical
maximum principle for local weak super(sub)-solutions to $(\ref{eq:equation})$
that are locally bounded from below (from above), commonly referred
to as De Giorgi-type Lemma, or Critical Mass Lemma. As a consequence
of this result, local weak super(sub)-solutions to (\ref{eq:equation})
admit a lower (upper) semicontinuous representative, provided they
are locally bounded from below (from above).\\
$\hspace*{1em}$Actually, in view of what has been established in
Sections \ref{sec:energy_estimate} and \ref{sec:boundedness}, analogous
results hold more generally for functions $u\in\mathcal{DG}^{\pm}(\mathbf{p},\{\delta_{i}\},\Omega_{T},\mathcal{C})$,
under suitable assumptions on the parameters $p_{i}$, and additionally
assuming that condition (\ref{eq:extra_integr}) holds in the case
$\overline{p}\leq2n/(n+2)$. However, for the sake of simplicity,
here we restrict ourselves to stating the aforementioned results only
for local weak super(sub)-solutions to (\ref{eq:equation}). The corresponding
statements for functions in $\mathcal{DG}^{\pm}(\mathbf{p},\{\delta_{i}\},\Omega_{T},\mathcal{C})$
can be easily deduced by the reader.\\
$\hspace*{1em}$In the sequel, we will use the cubes and cylinders
defined below, which turn out to be particularly convenient in our
anisotropic setting. For any $(y,s)\in\Omega_{T}$ and any $M,\rho>0$,
the \textit{intrinsic cubes} $\mathcal{K}_{\rho}(M)$, $[y+\mathcal{K}_{\rho}(M)]$
and the \textit{backward intrinsic cylinders} $\mathcal{Q}_{\rho}^{-}(M)$,
$(y,s)+\mathcal{Q}_{\rho}^{-}(M)$ are defined respectively by
\[
\mathcal{K}_{\rho}(M):=\,\prod_{i=1}^{n}\left\{ \vert x_{i}\vert<M^{\frac{p_{i}-2}{p_{i}}}\,\rho^{\frac{\overline{p}}{p_{i}}}\right\} ,\,\,\,\,\,\,\,\,\,\,\,\,\,\,\,\,\,\,\,\,[y+\mathcal{K}_{\rho}(M)]:=\,\prod_{i=1}^{n}\left\{ \vert x_{i}-y_{i}\vert<M^{\frac{p_{i}-2}{p_{i}}}\,\rho^{\frac{\overline{p}}{p_{i}}}\right\} 
\]
and
\[
\mathcal{Q}_{\rho}^{-}(M):=\,\mathcal{K}_{\rho}(M)\times(-\rho^{\overline{p}},0]\,,\,\,\,\,\,\,\,\,\,\,\,\,\,\,\,\,\,\,\,\,(y,s)+\mathcal{Q}_{\rho}^{-}(M):=\,[y+\mathcal{K}_{\rho}(M)]\times(s-\rho^{\overline{p}},s]\,.
\]
Throughout this section, $(y,s)\in\Omega_{T}$ and $M,\rho>0$ are
such that $(y,s)+\mathcal{Q}_{\rho}^{-}(M)\Subset\Omega_{T}$, while
$\eta\in C^{\infty}(\mathbb{R}^{n+1})$ is a smooth cut-off function
of the form 
\[
\eta(x,t):=\,\Psi(t)\,\prod_{i=1}^{n}\eta_{i}^{p_{i}}(x_{i})\,,
\]
where $\Psi\in C^{\infty}(\mathbb{R};[0,1])$ is a non-decreasing
map such that
\[
\Psi\equiv0\,\,\,\,\,\mathrm{on}\,\,\left(-\infty,s-\rho^{\overline{p}}+\frac{\epsilon}{2}\right]\,,\,\,\,\,\,\,\,\,\Psi\equiv1\,\,\,\,\,\mathrm{on}\,\,\,[s-\rho^{\overline{p}}+\epsilon,\infty)\,,\,\,\,\,\,\mathrm{with\,\,\,}\epsilon\in(0,\rho^{\overline{p}})\,,
\]
while the functions $\eta_{i}$ satisfy 
\[
\eta_{i}\,\in\,C_{0}^{\infty}(\mathbb{R};[0,1]),\,\,\,\,\,\,\,\,\mathrm{supp}\,\eta_{i}\subset\left(y_{i}-M^{\frac{p_{i}-2}{p_{i}}}\,\rho^{\frac{\overline{p}}{p_{i}}},\,y_{i}+M^{\frac{p_{i}-2}{p_{i}}}\,\rho^{\frac{\overline{p}}{p_{i}}}\right)\,\,\,\,\,\,\,\,\mathrm{for\,\,every}\,\,i\in\{1,\ldots,n\}.
\]
As an immediate consequence of these conditions, we have that
\[
\eta\equiv0\,\,\,\,\,\,\mathrm{on\,\,the\,\,parabolic\,\,boundary\,\,of\,\,}(y,s)+\mathcal{Q}_{\rho}^{-}(M)\,.
\]
The proof of the following two energy estimates is entirely analogous
to that of Proposition \ref{prop:PropEnergy}, and is therefore omitted.\vspace{0.5mm}

\begin{prop}
\noindent Let $n\geq2$ and $p_{1},\ldots,p_{n}>1$. Assume that $(\ref{eq:coeff})$
holds and that $u$ is a local weak supersolution to $(\ref{eq:equation})$
in the sense of Definition \ref{def:weaksoldef}. Then there exists
a positive constant $C$, depending only on $\Lambda$ and $\max\,\{p_{1},\ldots,p_{n}\}$,
such that for every $k\in\mathbb{R}$ we have\begin{align}\label{eq:enest_sup}
&\underset{\tau\,\in\,(s-\rho^{\overline{p}},s]}{\sup}\int_{[y\,+\,\mathcal{K}_{\rho}(M)]}(u-k)_{-}^{2}\,\eta(x,\tau)\,dx\,+\,\sum_{i=1}^{n}\iint_{(y,s)\,+\,\mathcal{Q}_{\rho}^{-}(M)}(\vert \partial_{x_{i}}u\vert-\delta_{i})_{+}^{p_{i}}\,\eta\,\mathds{1}_{\{u\,<\,k\}}\,dx\,dt\nonumber\\
&\,\,\,\,\,\,\,\leq\,C\iint_{(y,s)\,+\,\mathcal{Q}_{\rho}^{-}(M)}(u-k)_{-}^{2}\,\partial_{t}\eta\,dx\,dt\,+\,C\,\sum_{i=1}^{n}\iint_{(y,s)\,+\,\mathcal{Q}_{\rho}^{-}(M)}(u-k)_{-}^{p_{i}}\,\vert\partial_{x_{i}}\eta^{\frac{1}{p_{i}}}\vert^{p_{i}}\,dx\,dt\,.
\end{align}Similarly, if $u$ is a local weak subsolution to $(\ref{eq:equation})$,
then for every $k\in\mathbb{R}$ we have\begin{align}\label{eq:enest_subsol}
&\underset{\tau\,\in\,(s-\rho^{\overline{p}},s]}{\sup}\int_{[y\,+\,\mathcal{K}_{\rho}(M)]}(u-k)_{+}^{2}\,\eta(x,\tau)\,dx\,+\,\sum_{i=1}^{n}\iint_{(y,s)\,+\,\mathcal{Q}_{\rho}^{-}(M)}(\vert \partial_{x_{i}}u\vert-\delta_{i})_{+}^{p_{i}}\,\eta\,\mathds{1}_{\{u\,>\,k\}}\,dx\,dt\nonumber\\
&\,\,\,\,\,\,\,\leq\,C\iint_{(y,s)\,+\,\mathcal{Q}_{\rho}^{-}(M)}(u-k)_{+}^{2}\,\partial_{t}\eta\,dx\,dt\,+\,C\,\sum_{i=1}^{n}\iint_{(y,s)\,+\,\mathcal{Q}_{\rho}^{-}(M)}(u-k)_{+}^{p_{i}}\,\vert\partial_{x_{i}}\eta^{\frac{1}{p_{i}}}\vert^{p_{i}}\,dx\,dt\,.
\end{align}\vspace{0.5mm}
\end{prop}

\noindent $\hspace*{1em}$We are now in a position to prove the Critical
Mass Lemma mentioned above. We recall that local weak subsolutions
(resp. supersolutions) to (\ref{eq:equation}) are locally bounded
from above (resp. below) in $\Omega_{T}$, provided that additional
conditions are satisfied (see Remarks \ref{thm:Rk1} and \ref{thm:Rk2}).
Let us fix a cylinder $(y,s)+\mathcal{Q}_{2\rho}^{-}(M)\Subset\Omega_{T}$,
being $(y,s)\in\Omega_{T}$ and $M,\rho>0$ appropriate. Let $\mu^{+}$,
$\mu^{-}$ be such that 
\[
\mu^{-}\,\leq\,\underset{(y,s)\,+\,\mathcal{Q}_{2\rho}^{-}(M)}{\mathrm{ess}\,\inf}\,u\,\leq\,\underset{(y,s)\,+\,\mathcal{Q}_{2\rho}^{-}(M)}{\mathrm{ess}\,\sup}\,u\,\leq\,\mu^{+}\,.
\]
We also fix $a\in(0,1)$ and define 
\begin{equation}
\Gamma=\Gamma(\delta_{1},\ldots,\delta_{n},\mathbf{p}):=\,\sum_{i=1}^{n}\,\delta_{i}^{p_{i}}\,.\label{eq:Gamma}
\end{equation}

\begin{lem}[\textbf{De Giorgi-type/Critical Mass}]
\noindent \label{lem:critmass}Assume that $(\ref{eq:coeff})$ holds
and that $u$ is a local weak supersolution to $(\ref{eq:equation})$
in the sense of Definition \ref{def:weaksoldef}, locally bounded
from below. Let $\rho$, $M$, $\mu^{\pm}$, $a$ and $\Gamma$ be
as above. Then there exists $\nu^{-}\in(0,1)$, depending on the data
$\{n,\Lambda,\mathbf{p}\}$ and on the parameter $a$, but not on
the radius $\rho$ nor on $M$, such that if 
\[
\vert\{u\leq\mu^{-}+M\}\cap[(y,s)+\mathcal{Q}_{2\rho}^{-}(M)]\vert\,\leq\,\nu^{-}\,\vert\mathcal{Q}_{2\rho}^{-}(M)\vert
\]
and $M^{2}>\Gamma\,\rho^{\overline{p}}$, then 
\begin{equation}
u\,\geq\,\mu^{-}+aM\,\,\,\,\,\,\,\,\,\,\,\,\mathit{a.e}.\,\,\,\mathit{in}\,\,\,[(y,s)+\mathcal{Q}_{\rho}^{-}(M)]\,.\label{eq:tesi1}
\end{equation}
Likewise, if $u$ is a local weak subsolution to $(\ref{eq:equation})$
which is locally bounded from above, then there exists $\nu^{+}\in(0,1)$,
depending on the data $\{n,\Lambda,\mathbf{p}\}$ and on the parameter
$a$, but not on the radius $\rho$ nor on $M$, such that if
\[
\vert\{u\geq\mu^{+}-M\}\cap[(y,s)+\mathcal{Q}_{2\rho}^{-}(M)]\vert\,\leq\,\nu^{+}\,\vert\mathcal{Q}_{2\rho}^{-}(M)\vert
\]
and $M^{2}>\Gamma\,\rho^{\overline{p}}$, then 
\begin{equation}
u\,\leq\,\mu^{+}-aM\,\,\,\,\,\,\,\,\,\,\,\,\mathit{a.e}.\,\,\,\mathit{in}\,\,\,[(y,s)+\mathcal{Q}_{\rho}^{-}(M)]\,.\label{eq:tesi2}
\end{equation}
\end{lem}

\noindent \begin{proof}[\bfseries{Proof}]We prove (\ref{eq:tesi1}),
since the proof of (\ref{eq:tesi2}) is analogous. Without loss of
generality, we assume $(y,s)=(0,0)$, just to simplify the notation.
Let us set, for any $j,\ell\in\mathbb{N}_{0}$, 
\[
\rho_{j}:=\,\rho\,+\,\frac{\rho}{2^{j}}\,,\,\,\,\,\,\,\,\,\,\,\mathbb{K}_{j}:=\,\prod_{i=1}^{n}\left\{ \vert x_{i}\vert<M^{\frac{p_{i}-2}{p_{i}}}\,\rho^{\frac{\overline{p}}{p_{i}}}\left(1+\,\frac{1}{2^{j+\ell}}\right)\right\} ,\,\,\,\,\,\,\,\,\,\,\mathcal{Q}_{j}:=\,\mathbb{K}_{j}\times(-\rho_{j}^{\overline{p}},0]\,.
\]
Since $\mathbb{K}_{0}\to\overline{\mathcal{K}_{\rho}(M)}$ as $\ell\to\infty$,
we fix $\ell$ such that $\mathbb{K}_{0}\subset\mathcal{K}_{2\rho}(M)$.
Notice that $\ell$ can be chosen in such a way that it depends only
on $n$ and $\mathbf{p}$. We now introduce the sequence of levels
\[
k_{j}:=\,\mu^{-}+\,aM\,+\,\frac{(1-a)\,M}{2^{j}}\,,
\]
and note that 
\begin{equation}
(u-k_{j})_{-}\,\leq\left[aM\,+\,\frac{(1-a)\,M}{2^{j}}\right]\mathds{1}_{\{u\,<\,k_{j}\}}\,\leq\,M\,\mathds{1}_{\{u\,<\,k_{j}\}}\,\,\,\,\,\,\,\,\,\,\,\,\mathrm{a.e.}\,\,\,\mathrm{in}\,\,\,\mathcal{Q}_{2\rho}^{-}(M)\,.\label{eq:M_ineq}
\end{equation}
To move forward, we choose cut-off functions $\eta_{i,j}\in C_{0}^{\infty}(\mathbb{R};[0,1])$
such that 
\[
\eta_{i,j}\equiv1\,\,\,\,\,\mathrm{in}\,\,\,\pi_{i}(\mathbb{K}_{j+1})\,,\,\,\,\,\,\,\,\,\,\,\,\,\,\,\,\eta_{i,j}\equiv0\,\,\,\,\,\mathrm{in}\,\,\,\mathbb{R}\,\backslash\,\pi_{i}(\mathbb{K}_{j})\,\,\,\,\,\,\,\,\,\,\,\,\,\,\,\mathrm{and}\,\,\,\,\,\,\,\,\,\,\,\,\,\,\,\vert\eta'_{i,j}\vert\,\leq\,\frac{c\,2^{j}}{M^{\frac{p_{i}-2}{p_{i}}}\,\rho^{\frac{\overline{p}}{p_{i}}}}\,,
\]
where $c=c(n,\mathbf{p})>0$ and $\pi_{i}:\mathbb{R}^{n}\to\mathbb{R}$,
$\pi_{i}(x)=x_{i}$, $i\in\{1,\ldots,n\}$, denotes the canonical
projection onto the $i$-th spatial coordinate. Similarly, we take
$\Psi_{j}\in C^{\infty}(\mathbb{R};[0,1])$ such that 
\[
\Psi_{j}\equiv0\,\,\,\,\,\mathrm{on}\,\,\left(-\infty,-\,\,\frac{\rho_{j}^{\overline{p}}+\rho_{j+1}^{\overline{p}}}{2}\right]\,,\,\,\,\,\,\,\,\,\,\,\,\Psi_{j}\equiv1\,\,\,\,\,\mathrm{on}\,\,\,[-\rho_{j+1}^{\overline{p}},\infty)\,\,\,\,\,\,\,\,\,\,\,\mathrm{and}\,\,\,\,\,\,\,\,\,\,\,0\,\leq\,\Psi'_{j}\,\leq\,\frac{c\,2^{(j+1)\,\overline{p}}}{\rho^{\overline{p}}}\,.
\]
It follows that the smooth cut-off function 
\[
\eta_{j}(x,t):=\,\Psi_{j}(t)\,\prod_{i=1}^{n}\eta_{i,j}^{p_{i}}(x_{i})
\]

\noindent satisfies the following properties:
\begin{equation}
\begin{cases}
\begin{array}{c}
{\displaystyle 0\leq\eta_{j}\leq1\,,\,\,\,\,\,\,\,\,\,\,\eta_{j}\equiv1\,\,\,\,\,\,\mathrm{in}\,\,\,\mathcal{Q}_{j+1}\,,\,\,\,\,\,\,\,\,\,\,\eta_{j}\equiv0\,\,\,\,\,\,\mathrm{on}\,\,\,\partial_{\mathrm{par}}\mathcal{Q}_{j}\,,\,\,\,\,\,\,\,\,\,\,\,\,\,\,\,\,\,\,\,\,\,\,\,\,\,\,\,\,\,\,\,\,\,\,\,\,\,\,\,\,\,\,\,\,\,\,\,\,\,}\vspace{4mm}\\
{\displaystyle 0\,\leq\,\partial_{t}\eta_{j}\,\leq\,\frac{c\,2^{(j+1)\,\overline{p}}}{\rho^{\overline{p}}}\,\,,\,\,\,\,\,\,\,\,\,\,\vert\partial_{x_{i}}\eta_{j}^{\frac{1}{p_{i}}}\vert^{p_{i}}\,\leq\,\frac{c\,2^{j\,\mathcal{P}}}{M^{p_{i}-2}\,\rho^{\overline{p}}}\,\,\,\,\,\,\,\,\mathrm{for\,\,every}\,\,i\in\{1,\ldots,n\}.}
\end{array}\end{cases}\label{eq:eta}
\end{equation}
The construction of this cut-off function is similar to that of $\tilde{\zeta}_{j}$
in the proof of Theorem \ref{thm:main} and is obtained by suitably
modifying the calculations carried out in the \hyperref[sec:appendice]{Appendix}.
Now we rewrite estimate \eqref{eq:enest_sup} with $\rho_{j}$, $\mathbb{K}_{j}$,
$\mathcal{Q}_{j}$, $k_{j}$ and $\eta_{j}$ in place of $\rho$,
$\mathcal{K}_{\rho}(M)$, $\mathcal{Q}_{\rho}^{-}(M)$, $k$ and $\eta,$
respectively. We thus obtain\begin{align*}
&\underset{\tau\,\in\,(-\rho_{j}^{\overline{p}},0]}{\sup}\int_{\mathbb{K}_{j}}(u-k_{j})_{-}^{2}\,\eta_{j}(x,\tau)\,dx\,+\,\sum_{i=1}^{n}\iint_{\mathcal{Q}_{j}\,\cap\,\{u\,<\,k_{j}\}}(\vert \partial_{x_{i}}u\vert-\delta_{i})_{+}^{p_{i}}\,\eta_{j}\,dx\,dt\nonumber\\
&\,\,\,\,\,\,\,\leq\,\frac{C\,2^{j\,\overline{p}}}{\rho^{\overline{p}}}\iint_{\mathcal{Q}_{j}}(u-k_{j})_{-}^{2}\,dx\,dt\,+\,\sum_{i=1}^{n}\,\frac{C\,2^{j\,\mathcal{P}}}{M^{p_{i}-2}\,\rho^{\overline{p}}}\iint_{\mathcal{Q}_{j}}(u-k_{j})_{-}^{p_{i}}\,dx\,dt\,,
\end{align*}where $C$ is a positive constant depending only on $n$, $\Lambda$
and $\mathbf{p}$. At this stage, we set 
\[
\mathcal{A}_{j}:=\,\left\{ (x,t)\in\mathcal{Q}_{j}:u(x,t)<k_{j}\right\} .
\]
Using the properties of $\eta_{j}$ in (\ref{eq:eta}) and the definition
of $\Gamma$ in (\ref{eq:Gamma}), we get\begin{align*}
&\sum_{i=1}^{n}\iint_{\mathcal{Q}_{j}}\vert\partial_{x_{i}}[\eta_{j}\,(u-k_{j})_{-}]\vert^{p_{i}}\,dx\,dt\nonumber\\
&\,\,\,\,\,\,\,\leq\,2^{\mathcal{P}-1}\left[\sum_{i=1}^{n}\iint_{\mathcal{Q}_{j}}(u-k_{j})_{-}^{p_{i}}\,\vert\partial_{x_{i}}\eta_{j}\vert^{p_{i}}\,dx\,dt\,+\,\sum_{i=1}^{n}\iint_{\mathcal{A}_{j}}\eta_{j}\,\vert \partial_{x_{i}}u\vert^{p_{i}}\,dx\,dt\right]\nonumber\\
&\,\,\,\,\,\,\,\leq\,2^{2\mathcal{P}-2}\left[\sum_{i=1}^{n}\frac{c\,2^{j\,\mathcal{P}}}{M^{p_{i}-2}\,\rho^{\overline{p}}}\iint_{\mathcal{Q}_{j}}(u-k_{j})_{-}^{p_{i}}\,dx\,dt\,+\,\sum_{i=1}^{n}\iint_{\mathcal{A}_{j}}(\vert \partial_{x_{i}}u\vert-\delta_{i})_{+}^{p_{i}}\,\eta_{j}\,dx\,dt\,+\,\vert\mathcal{A}_{j}\vert\,\Gamma\right].
\end{align*}Putting together the two previous estimates, applying (\ref{eq:M_ineq})
and recalling that $M^{2}>\Gamma\,\rho^{\overline{p}}$, we deduce
that\begin{align*}
\mathcal{E}_{j}&:=\underset{\tau\,\in\,(-\rho_{j}^{\overline{p}},0]}{\sup}\int_{\mathbb{K}_{j}}(u-k_{j})_{-}^{2}\,\eta_{j}(x,\tau)\,dx\,+\,\sum_{i=1}^{n}\iint_{\mathcal{Q}_{j}}\vert\partial_{x_{i}}[\eta_{j}\,(u-k_{j})_{-}]\vert^{p_{i}}\,dx\,dt\nonumber\\
&\leq\,\frac{C\,2^{j\,\overline{p}}}{\rho^{\overline{p}}}\iint_{\mathcal{Q}_{j}}(u-k_{j})_{-}^{2}\,dx\,dt\,+\,\sum_{i=1}^{n}\,\frac{C\,2^{j\,\mathcal{P}}}{M^{p_{i}-2}\,\rho^{\overline{p}}}\iint_{\mathcal{Q}_{j}}(u-k_{j})_{-}^{p_{i}}\,dx\,dt\,+\,C\,\vert\mathcal{A}_{j}\vert\,\Gamma\nonumber\\
&\leq\,C\,2^{j\,\mathcal{P}}\,\vert\mathcal{A}_{j}\vert\left(\frac{M^{2}}{\rho^{\overline{p}}}\,+\,\Gamma\right)\leq\,\frac{C\,2^{j\,\mathcal{P}\,+\,1}\,\vert\mathcal{A}_{j}\vert\,M^{2}}{\rho^{\overline{p}}}\,.
\end{align*}Now we combine these estimates of the energies $\mathcal{E}_{j}$
with the following embedding inequality (see \cite{DMV})\begin{align*}
&\iint_{\mathcal{Q}_{j}}\vert\eta_{j}\,(u-k_{j})_{-}\vert^{\overline{p}\left(\frac{n+2}{n}\right)}\,dx\,dt\\
&\,\,\,\,\,\,\,\leq\,c\left[\underset{\tau\,\in\,(-\rho_{j}^{\overline{p}},0]}{\sup}\int_{\mathbb{K}_{j}}(u-k_{j})_{-}^{2}\,\eta_{j}^{2}(x,\tau)\,dx\,+\,\sum_{i=1}^{n}\iint_{\mathcal{Q}_{j}}\vert\partial_{x_{i}}[\eta_{j}\,(u-k_{j})_{-}]\vert^{p_{i}}\,dx\,dt\right]^{\frac{n\,+\,\overline{p}}{n}}\leq\,c\,\mathcal{E}_{j}^{\frac{n\,+\,\overline{p}}{n}}.
\end{align*}Observing that $(u-k_{j})_{-}\geq(k_{j}-k_{j+1})=(1-a)M/2^{j+1}$
in $\mathcal{A}_{j+1}$, and recalling that $\eta_{j}=1$ in $\mathcal{Q}_{j+1}$,
we get the chain\begin{align}\label{eq:cml_1}
\left[\frac{(1-a)M}{2^{j+1}}\right]^{\overline{p}\left(\frac{n+2}{n}\right)}\vert\mathcal{A}_{j+1}\vert\,&\leq\,\iint_{\mathcal{A}_{j+1}}(u-k_{j})_{-}^{\overline{p}\left(\frac{n+2}{n}\right)}\,dx\,dt\nonumber\\
&\le\,\iint_{\mathcal{Q}_{j}}\vert\eta_{j}\,(u-k_{j})_{-}\vert^{\overline{p}\left(\frac{n+2}{n}\right)}\,dx\,dt\,\leq\,c\,\mathcal{E}_{j}^{\frac{n\,+\,\overline{p}}{n}}\nonumber\\
&\leq\,\frac{C\,2^{bj}\,M^{\frac{2\left(n\,+\,\overline{p}\right)}{n}}}{\rho^{\frac{\overline{p}\left(n\,+\,\overline{p}\right)}{n}}}\,\vert\mathcal{A}_{j}\vert^{1\,+\,\frac{\overline{p}}{n}}\,,
\end{align}where $b=b(n,\mathbf{p})>1$. At this point, we set
\[
Y_{j}:=\,\frac{\vert\mathcal{A}_{j}\vert}{\vert\mathcal{Q}_{j}\vert}
\]
and observe that $\vert\mathcal{Q}_{j}\vert^{\frac{\overline{p}}{n}}\leq c(n,\mathbf{p})\,M^{\overline{p}\,-\,2}\rho^{\frac{\overline{p}}{n}\left(n\,+\,\overline{p}\right)}$.
Dividing both sides of \eqref{eq:cml_1} by $\vert\mathcal{Q}_{j+1}\vert$,
and noticing that $\vert\mathcal{Q}_{j+1}\vert\geq2^{-\left(n\,+\,\overline{p}\right)}\,\vert\mathcal{Q}_{j}\vert$,
we obtain\begin{align}\label{eq:cml_2}
Y_{j+1}\,&\leq\,\frac{C\,2^{bj}\,M^{2\,-\,\overline{p}}}{(1-a)^{\frac{\overline{p}\left(n\,+\,2\right)}{n}}\,\rho^{\frac{\overline{p}\left(n\,+\,\overline{p}\right)}{n}}}\,\,\frac{\vert\mathcal{A}_{j}\vert^{1\,+\,\frac{\overline{p}}{n}}}{\vert\mathcal{Q}_{j}\vert}\nonumber\\
&\leq\,\frac{C\,2^{bj}}{(1-a)^{\frac{\overline{p}\left(n\,+\,2\right)}{n}}}\,Y_{j}^{1\,+\,\frac{\overline{p}}{n}}\,=:\,\gamma\,2^{bj}\,Y_{j}^{1\,+\,\frac{\overline{p}}{n}}\,,
\end{align}where $\gamma$ is a positive constant depending only on $n$, $\Lambda$,
$\mathbf{p}$ and $a$. In the previous estimate, we may assume that
$C>1$; hence, $\gamma>1$, since $a\in(0,1)$. According to Lemma
\ref{lem:Giusti}, the estimate \eqref{eq:cml_2} yields $Y_{\infty}:=\lim_{j\to\infty}Y_{j}=0$
provided that 
\begin{equation}
Y_{0}:=\,\frac{\vert\{u<\mu^{-}+M\}\cap\mathcal{Q}_{0}\vert}{\vert\mathcal{Q}_{0}\vert}\,\leq\,\gamma^{-\,\frac{n}{\overline{p}}}\,2^{-\,\frac{b\,n^{2}}{\overline{p}^{2}}}\,=:\,\nu^{*}.\label{eq:prem1}
\end{equation}
Then $Y_{\infty}=0$ implies (\ref{eq:tesi1}). It remains to prove
that (\ref{eq:prem1}) holds.\\
We set $\nu^{-}:=\nu^{*}/\gamma^{*}$, where $\nu^{*}\in(0,1)$ is
defined in (\ref{eq:prem1}) and $\gamma^{*}$ is such that $\vert\mathcal{Q}_{2\rho}^{-}(M)\vert\leq\gamma^{*}\,\vert\mathcal{Q}_{0}\vert$.
Therefore, $\gamma^{*}\ge1$ and 
\[
\vert\{u<\mu^{-}+M\}\cap\mathcal{Q}_{0}\vert\,\leq\,\vert\{u\leq\mu^{-}+M\}\cap\mathcal{Q}_{2\rho}^{-}(M)\vert\,\leq\,\nu^{-}\,\vert\mathcal{Q}_{2\rho}^{-}(M)\vert\,\leq\,\nu^{*}\,\vert\mathcal{Q}_{0}\vert\,,
\]
which ensures (\ref{eq:prem1}).\\
$\hspace*{1em}$In order to prove (\ref{eq:tesi2}), we may proceed
as above, considering the energy estimates \eqref{eq:enest_subsol}
in the same iterative geometry, but this time using the truncations
$(u-k_{j})_{+}$ with 
\[
k_{j}\,=\,\mu^{+}-\,aM\,-\,\frac{(1-a)\,M}{2^{j}}\,.
\]
This concludes the proof.\end{proof}

\noindent $\hspace*{1em}$At this point, we are ready to state the
following result.
\begin{thm}
\noindent \label{thm:semicontinuity}Assume that $(\ref{eq:coeff})$
holds and that $u$ is a local weak supersolution to $(\ref{eq:equation})$
in the sense of Definition \ref{def:weaksoldef}, locally bounded
from below. Then $u$ has a lower semicontinuous representative in
$\Omega_{T}$.\\
Likewise, if $u$ is a local weak subsolution to $(\ref{eq:equation})$
which is locally bounded from above, then $u$ has an upper semicontinuous
representative in $\Omega_{T}$.\vspace{-3mm}
\end{thm}

\noindent \begin{proof}[\bfseries{Proof}] We proceed in a way reminiscent
of \cite{CiaGuaVes} and \cite{Liao}. Let $u$ be a local weak supersolution
to (\ref{eq:equation}) in the sense of Definition \ref{def:weaksoldef},
locally bounded from below. Set $\mathcal{Q}_{\rho}:=\mathcal{Q}_{\rho}^{-}(1)$
for all $\rho>0$, and consider the lower semicontinuous regularization
of $u$, defined by
\[
u_{*}(x,t):=\lim_{\rho\to0^{+}}\underset{(x,t)\,+\,\mathcal{Q}_{\rho}}{\mathrm{ess}\,\inf}\,u\,\,\,\,\,\,\,\,\,\,\,\,\,\mathrm{for}\,\,(x,t)\in\Omega_{T}\,.
\]
We observe that this function is well defined at every point of $\Omega_{T}$,
since $u$ is locally bounded from below and $(x,t)+\mathcal{Q}_{\rho}\subset\Omega_{T}$
for small values of $\rho$. It is well known that $u_{*}$ is lower
semicontinuous. Therefore, showing that $u=u_{*}$ almost everywhere
in $\Omega_{T}$ yields the desired conclusion. In order to prove
this equality, we also define the set 
\[
\mathcal{L}:=\left\{ (x,t)\in\Omega_{T}:\,\vert u(x,t)\vert<\infty,\,\,\lim_{\rho\to0^{+}}\tiltfiint_{(x,t)\,+\,\mathcal{Q}_{\rho}}\vert u(x,t)-u(y,s)\vert\,dy\,ds=0\right\} .
\]
This set is well defined, since $u\in L_{loc}^{1}(\Omega_{T})$. Moreover,
\begin{equation}
\vert\mathcal{L}\vert=\vert\Omega_{T}\vert\label{eq:same_meas}
\end{equation}
(see \cite[Section 4]{CiaGuaVes}). Taking (\ref{eq:same_meas}) into
account, it is sufficient to prove $u=u_{*}$ in $\mathcal{L}$. For
all $(x,t)\in\mathcal{L}$, we have 
\[
u_{*}(x,t)\,=\,\lim_{\rho\to0^{+}}\underset{(x,t)\,+\,\mathcal{Q}_{\rho}}{\mathrm{ess}\,\inf}\,u\,\leq\,\lim_{\rho\to0^{+}}\tiltfiint_{(x,t)\,+\,\mathcal{Q}_{\rho}}u(y,s)\,dy\,ds\,=\,u(x,t)\,.
\]
To show the reverse inequality, we argue by contradiction. Let us
take $(x_{0},t_{0})\in\mathcal{L}$ and suppose that $u_{*}(x_{0},t_{0})<u(x_{0},t_{0})$.
Let $r,b>0$ be sufficiently small so that $(x_{0},t_{0})+\mathcal{Q}_{r}\subset\Omega_{T}$
and 
\[
\underset{(x_{0},t_{0})\,+\,\mathcal{Q}_{r}}{\mathrm{ess}\,\inf}\,u\,=:\,\mu_{-}\,\leq\,u_{*}(x_{0},t_{0})\,<\,\mu_{-}+b\,<\,u(x_{0},t_{0})\,.
\]
This choice is possible by the very definition of $u_{*}$, since
$u_{*}(x_{0},t_{0})$ is close to $\mu_{-}$ for small values of $r$.
Let us introduce $a\in(0,1)$ such that
\[
\mu_{-}+ab\,>\,u_{*}(x_{0},t_{0})\,,\,\,\,\,\,\,\,\,\,\,\,\,\mathrm{i.e.},\,\,\,\,\,\,\,\,\,\,\,\,\frac{u_{*}(x_{0},t_{0})-\mu_{-}}{b}\,<\,a\,<\,1\,.
\]
Now we set 
\[
\rho^{*}:=\,\inf\left\{ \left(\frac{b^{2}}{\Gamma}\right)^{\frac{1}{\overline{p}}},\,\frac{r}{2}\,,\,\min_{1\,\leq\,i\,\leq\,n}\left\{ b^{\frac{2-p_{i}}{\overline{p}}}\,\frac{r}{2}\right\} \right\} ,
\]
where $\Gamma$ is defined in (\ref{eq:Gamma}). If $\Gamma=0$ (which
occurs only when $\delta_{i}=0$ for all $i\in\{1,...,n\}$), the
quantity $(b^{2}/\Gamma)^{1/\overline{p}}$ is understood as $+\infty$.
With the above choice of $\rho^{*}$, for every $\rho\in(0,\rho^{*})$
one has 
\[
b^{2}>\Gamma\,\rho^{\overline{p}}\,\,\,\,\,\,\,\,\,\,\,\,\mathrm{and}\,\,\,\,\,\,\,\,\,\,\,\,\mathcal{Q}_{2\rho}^{-}(b)\,\subset\,\mathcal{Q}_{r}\,.
\]
Therefore, for every $\rho\in(0,\rho^{*})$ we have 
\[
\mu_{-}\,\leq\,\underset{(x_{0},t_{0})\,+\,\mathcal{Q}_{2\rho}^{-}(b)}{\mathrm{ess}\,\inf}\,u\,.
\]
$\hspace*{1em}$At this point, we prove that there exists $\nu_{a}^{-}>0$,
depending only on $a$, $n$, $\Lambda$, $\mathbf{p}$, such that
for some $\tilde{\rho}\in(0,\rho^{*})$ one has 
\begin{equation}
\vert\{u\leq\mu_{-}+b\}\cap[(x_{0},t_{0})+\mathcal{Q}_{2\tilde{\rho}}^{-}(b)]\vert\,\leq\,\nu_{a}^{-}\,\vert\mathcal{Q}_{2\tilde{\rho}}^{-}(b)\vert\,.\label{eq:meas_cond}
\end{equation}
If this were not the case, then for every $\rho\in(0,\rho^{*})$ we
would have\begin{align*}
\iint_{(x_{0},t_{0})\,+\,\mathcal{Q}_{2\rho}^{-}(b)}\vert u(x_{0},t_{0})-u(x,t)\vert\,dx\,dt\,&\geq\,\iint_{\{u\,\leq\,\mu_{-}+\,b\}\,\cap\,[(x_{0},t_{0})\,+\,\mathcal{Q}_{2\rho}^{-}(b)]}[u(x_{0},t_{0})-(\mu_{-}+b)]\,dx\,dt\nonumber\\
&\geq\,\nu_{a}^{-}\,[u(x_{0},t_{0})-(\mu_{-}+b)]\,\vert\mathcal{Q}_{2\rho}^{-}(b)\vert\,,
\end{align*} which, after dividing by $\vert\mathcal{Q}_{2\rho}^{-}(b)\vert$,
yields
\begin{equation}
\tiltfiint_{(x_{0},t_{0})\,+\,\mathcal{Q}_{2\rho}^{-}(b)}\vert u(x_{0},t_{0})-u(x,t)\vert\,dx\,dt\,\geq\,\nu_{a}^{-}\,[u(x_{0},t_{0})-(\mu_{-}+b)]\,>0\,,\label{eq:assurdo}
\end{equation}
for all $\rho\in(0,\rho^{*})$. Setting
\[
\kappa=\kappa(n,\mathbf{p},b):=\,\max\left\{ 2,\max_{1\,\leq\,i\,\leq\,n}\left\{ 2\,b^{\frac{p_{i}-2}{\overline{p}}}\right\} \right\} ,
\]
it is easy to verify that $\mathcal{Q}_{2\rho}^{-}(b)\subseteq\mathcal{Q}_{\kappa\rho}$,
and hence
\begin{equation}
\tiltfiint_{(x_{0},t_{0})\,+\,\mathcal{Q}_{2\rho}^{-}(b)}\vert u(x_{0},t_{0})-u\vert\,dx\,dt\,\leq\,\frac{\vert\mathcal{Q}_{\kappa\rho}\vert}{\vert\mathcal{Q}_{2\rho}^{-}(b)\vert}\,\tiltfiint_{(x_{0},t_{0})\,+\,\mathcal{Q}_{\kappa\rho}}\vert u(x_{0},t_{0})-u\vert\,dx\,dt\,.\label{eq:chiave}
\end{equation}
Observe that 
\[
\frac{\vert\mathcal{Q}_{\kappa\rho}\vert}{\vert\mathcal{Q}_{2\rho}^{-}(b)\vert}\,=\,\max\left\{ 1,\underset{1\,\leq\,i\,\leq\,n}{\max}\left\{ b^{\frac{\left(p_{i}-2\right)\left(n+\overline{p}\right)}{\overline{p}}}\right\} \right\} \,\prod_{j=1}^{n}b^{\frac{2\,-\,p_{j}}{p_{j}}}\,,
\]
i.e., this ratio depends only on $n$, $\mathbf{p}$, $b$, but not
on the radius $\rho$. Then, from (\ref{eq:chiave}) and the fact
that $(x_{0},t_{0})\in\mathcal{L}$, it follows that 
\[
\lim_{\rho\to0^{+}}\tiltfiint_{(x_{0},t_{0})\,+\,\mathcal{Q}_{2\rho}^{-}(b)}\vert u(x_{0},t_{0})-u(x,t)\vert\,dx\,dt\,=0\,,
\]
which contradicts (\ref{eq:assurdo}). Therefore, condition (\ref{eq:meas_cond})
holds for some $\tilde{\rho}\in(0,\rho^{*})$, and we may apply Lemma
\ref{lem:critmass} to deduce that
\[
u(x,t)\,\geq\,\mu_{-}+ab\,>\,u_{*}(x_{0},t_{0})\,,\,\,\,\,\,\,\,\,\,\,\,\,\mathrm{for}\,\,\mathit{\mathrm{a.e}}.\,\,(x,t)\in[(x_{0},t_{0})+\mathcal{Q}_{\tilde{\rho}}^{-}(b)]\,.
\]
This however contradicts the definition of $u_{*}(x_{0},t_{0})$,
since
\[
u_{*}(x_{0},t_{0})\,=\,\lim_{\rho\to0^{+}}\underset{(x_{0},t_{0})\,+\,\mathcal{Q}_{\rho}}{\mathrm{ess}\,\inf}\,u\,\geq\,\underset{(x_{0},t_{0})\,+\,\mathcal{Q}_{\tilde{\rho}}^{-}(b)}{\mathrm{ess}\,\inf}\,u\,>\,u_{*}(x_{0},t_{0})\,.
\]
As a result, we must have $u_{*}(x,t)\geq u(x,t)$ for all $(x,t)\in\mathcal{L}$.\\
$\hspace*{1em}$Finally, the second part of the statement follows
from the observation that if $u$ is a local weak subsolution to (\ref{eq:equation}),
then $-u$ is a local weak supersolution, together with the fact that
the additive inverse of a lower semicontinuous function is upper semicontinuous.\end{proof}

\noindent \begin{brem}Theorem \ref{thm:semicontinuity} implies that
the corresponding semicontinuous representatives of $u$ actually
satisfy the pointwise bound (\ref{eq:tesi1}) or (\ref{eq:tesi2})
for every point in the intrinsic cylinder $(y,s)+\mathcal{Q}_{\rho}^{-}(M)$.\end{brem}

\noindent \begin{brem}Let $u$ be a local weak solution to (\ref{eq:equation})
in the sense of Definition \ref{def:locweaksol2}, and assume that
$u$ is locally bounded in $\Omega_{T}$. By Theorem \ref{thm:semicontinuity},
$u$ admits a lower semicontinuous representative $u_{*}$ and an
upper semicontinuous representative $\tilde{u}$. However, this result
does not, in general, imply that $u$ is continuous on $\Omega_{T}$.
Indeed, one cannot exclude that 
\[
u_{*}(x,t)\,<\,\tilde{u}(x,t)
\]

\noindent on a nonempty set of Lebesgue measure zero. Therefore, continuity
of $u$ cannot be deduced solely from the existence of these semicontinuous
representatives. In this regard, we aim to address the continuity
of local weak solutions in the case $p_{1}=\cdots=p_{n}$ in future
work.\end{brem}

\noindent \appendix
\section{Appendix}\label{sec:appendice}$\hspace*{1em}$In what follows, we first show how to explicitly construct
a cut-off function $\zeta$ of the form (\ref{eq:cut-off}), satisfying
(\ref{eq:psi}) and (\ref{eq:zeta_i}) with $(x_{0},t_{0})=(0,0)$.
Then, we construct a smooth cut-off function $\tilde{\zeta}_{j}$
having the properties stated in (\ref{eq:properties}).\medskip{}

\noindent \textbf{Construction of the cut-off function $\zeta$.}
Let $\chi\in C^{\infty}(\mathbb{R})$ be a non-decreasing function
such that 
\[
0\leq\chi\leq1\,,\,\,\,\,\,\,\,\,\chi\equiv0\,\,\,\,\,\mathrm{on}\,\,\,(-\infty,0]\,,\,\,\,\,\,\,\,\,\chi\equiv1\,\,\,\,\,\mathrm{on}\,\,\,[1,\infty)\,.
\]
For example, one can take 
\[
\chi(t)\,=\,\frac{Z(t)}{Z(t)+Z(1-t)}\,,\,\,\,\,\,\,\,\,\mathrm{with}\,\,\,\,\,Z(t):=\begin{cases}
\begin{array}{cc}
{\displaystyle e^{-\,\frac{1}{t}}} & \mathrm{if}\,\,t>0,\\
0 & \mathrm{if}\,\,t\le0.
\end{array}\end{cases}
\]
As an immediate consequence of Lagrange's mean value theorem, we have
that $\Vert\chi'\Vert_{\infty}\geq1$.

\noindent For $i\in\{1,\ldots,n\}$ and $\rho\in(0,\infty)$, we choose
three positive numbers $a_{i}$, $b_{i}$ and $\epsilon$ such that
\[
0<b_{i}<a_{i}<\rho^{\frac{1}{p_{i}}}\,\,\,\,\,\,\,\,\,\,\,\,\mathrm{and}\,\,\,\,\,\,\,\,\,\,\,\epsilon\in(0,\rho)\,.
\]
Furthermore, we define 
\[
\psi(t):=\,\chi\left(\frac{2t+2\rho-\epsilon}{\epsilon}\right),\,\,\,\,\,\,\,\,t\in\mathbb{R}\,,
\]
and 
\[
\zeta_{i}(s):=\,\chi\left(\frac{s+a_{i}}{a_{i}-b_{i}}\right)\,\chi\left(\frac{a_{i}-s}{a_{i}-b_{i}}\right),\,\,\,\,\,\,\,\,s\in\mathbb{R}\,.
\]
\\
One can easily check that $\psi\in C^{\infty}(\mathbb{R};[0,1])$
is a non-decreasing map such that
\[
\psi\equiv0\,\,\,\,\,\mathrm{on}\,\,\left(-\infty,-\rho+\frac{\epsilon}{2}\right]\,,\,\,\,\,\,\,\,\,\psi\equiv1\,\,\,\,\,\mathrm{on}\,\,\,[-\rho+\epsilon,\infty)\,,
\]
while the functions $\zeta_{i}$ satisfy 
\[
\zeta_{i}\equiv1\,\,\,\,\,\mathrm{on}\,\,\,[-b_{i},b_{i}]\,,\,\,\,\,\,\,\,\,\,\,\zeta_{i}\equiv0\,\,\,\,\,\mathrm{on}\,\,\,\mathbb{R}\,\backslash\,(-a_{i},a_{i})\,.
\]
Therefore, 
\[
\zeta_{i}\,\in\,C_{0}^{\infty}(\mathbb{R};[0,1])\,\,\,\,\,\,\,\,\,\,\mathrm{and}\,\,\,\,\,\,\,\,\,\,\mathrm{supp}\,\zeta_{i}\subset(-\rho^{\frac{1}{p_{i}}},\rho^{\frac{1}{p_{i}}})\,\,\,\,\,\,\,\,\mathrm{for\,\,every}\,\,i\in\{1,\ldots,n\}.
\]
$\hspace*{1em}$Now we consider the smooth cut-off function
\[
\zeta(x,t):=\,\psi(t)\,\prod_{i=1}^{n}\zeta_{i}^{p_{i}}(x_{i})\,,\,\,\,\,\,\,\,\,\,\,\,\,(x,t)\in\mathbb{R}^{n+1}.
\]
By the properties of $\psi$ and $\zeta_{1},\ldots,\zeta_{n}$, we
immediately deduce that 
\[
0\leq\zeta\leq1\,,\,\,\,\,\,\,\,\,\,\,\zeta\equiv0\,\,\,\,\,\,\mathrm{on}\,\,\,\partial_{\mathrm{par}}Q_{\rho}\,,\,\,\,\,\,\,\,\,\,\,\zeta\equiv1\,\,\,\,\,\,\mathrm{in}\,\,\,\prod_{i=1}^{n}(-b_{i},b_{i})\times(-\rho+\epsilon,0)\,.
\]
\\
\textbf{Construction of the cut-off function $\tilde{\zeta}_{j}$.}
For $i\in\{1,\ldots,n\}$ and a fixed $j\in\mathbb{N}_{0}$, we set
\[
a_{i,j}:=\,(\tilde{\rho}_{j})^{\frac{1}{p_{i}}}\,,\,\,\,\,\,\,\,\,\,\,\,\,b_{i,j}:=\,\rho_{j+1}^{\frac{1}{p_{i}}}\,,
\]
where $\rho_{j+1}$ and $\tilde{\rho}_{j}$ are defined in (\ref{eq:radii})
and (\ref{eq:radii2}), respectively. Thus we have $0<b_{i,j}<a_{i,j}<\rho_{j}^{\frac{1}{p_{i}}}$.
Moreover, we define 
\begin{equation}
\psi_{j}(t):=\,\chi\left(\frac{t+\tilde{\rho}_{j}}{\tilde{\rho}_{j}-\rho_{j+1}}\right),\,\,\,\,\,\,\,\,t\in\mathbb{R}\,,\label{eq:psij}
\end{equation}
and 
\begin{equation}
\tilde{\zeta}_{i,j}(s):=\,\chi\left(\frac{s+a_{i,j}}{a_{i,j}-b_{i,j}}\right)\,\chi\left(\frac{a_{i,j}-s}{a_{i,j}-b_{i,j}}\right),\,\,\,\,\,\,\,\,s\in\mathbb{R}\,.\label{eq:zetai_espl}
\end{equation}
\\
Then, $\psi_{j}\in C^{\infty}(\mathbb{R};[0,1])$ is a non-decreasing
map such that 
\[
\psi_{j}\equiv0\,\,\,\,\,\mathrm{on}\,\,(-\infty,-\tilde{\rho}_{j}]\,,\,\,\,\,\,\,\,\,\psi_{j}\equiv1\,\,\,\,\,\mathrm{on}\,\,\,[-\rho_{j+1},\infty)\,,
\]
while the functions $\tilde{\zeta}_{i,j}$ satisfy 
\[
\tilde{\zeta}_{i,j}\equiv1\,\,\,\,\,\mathrm{on}\,\,\,[-b_{i,j},b_{i,j}]\,,\,\,\,\,\,\,\,\,\,\,\tilde{\zeta}_{i,j}\equiv0\,\,\,\,\,\mathrm{on}\,\,\,\mathbb{R}\,\backslash\,(-a_{i,j},a_{i,j})\,.
\]
Therefore, 
\[
\tilde{\zeta}_{i,j}\,\in\,C_{0}^{\infty}(\mathbb{R};[0,1])\,\,\,\,\,\,\,\,\,\,\mathrm{and}\,\,\,\,\,\,\,\,\,\,\mathrm{supp}\,\tilde{\zeta}_{i,j}\subset(-\rho_{j}^{\frac{1}{p_{i}}},\rho_{j}^{\frac{1}{p_{i}}})\,\,\,\,\,\,\,\,\mathrm{for\,\,every}\,\,i\in\{1,\ldots,n\}.
\]
$\hspace*{1em}$Now we consider the smooth cut-off function
\[
\tilde{\zeta}_{j}(x,t):=\,\psi_{j}(t)\,\prod_{i=1}^{n}\tilde{\zeta}_{i,j}^{p_{i}}(x_{i})\,,\,\,\,\,\,\,\,\,\,\,\,\,(x,t)\in\mathbb{R}^{n+1}.
\]
By the properties of $\psi_{j}$ and $\tilde{\zeta}_{1,j},\ldots,\tilde{\zeta}_{n,j}$,
we immediately have that 
\[
0\leq\tilde{\zeta}_{j}\leq1\,,\,\,\,\,\,\,\,\,\,\,\tilde{\zeta}_{j}\equiv1\,\,\,\,\,\,\mathrm{in}\,\,\,\mathcal{Q}_{j+1}\,,\,\,\,\,\,\,\,\,\,\,\tilde{\zeta}_{j}\equiv0\,\,\,\,\,\,\mathrm{on}\,\,\,\partial_{\mathrm{par}}\mathcal{Q}_{j}\,,
\]
where the cylinders $\mathcal{Q}_{j}$ and $\mathcal{Q}_{j+1}$ are
defined as in the proof of Theorem \ref{thm:main}. Furthermore, by
(\ref{eq:psij}) and the definitions of $\rho_{j+1}$, $\tilde{\rho}_{j}$,
we have 
\[
\psi_{j}'(t)=\,\frac{1}{\tilde{\rho}_{j}-\rho_{j+1}}\,\,\chi'\left(\frac{t+\tilde{\rho}_{j}}{\tilde{\rho}_{j}-\rho_{j+1}}\right)=\,\frac{2^{j+2}}{(1-\sigma)\,\rho}\,\,\chi'\left(\frac{t+\tilde{\rho}_{j}}{\tilde{\rho}_{j}-\rho_{j+1}}\right),
\]
and therefore 
\begin{equation}
0\,\leq\,\partial_{t}\tilde{\zeta}_{j}(x,t)\,=\,\psi_{j}'(t)\,\prod_{i=1}^{n}\tilde{\zeta}_{i,j}^{p_{i}}(x_{i})\,\leq\,\frac{2^{j+2}\,\Vert\chi'\Vert_{\infty}}{(1-\sigma)\,\rho}\,\,\,\,\,\,\,\,\,\,\,\,\mathrm{for\,\,every\,\,}(x,t)\in\mathbb{R}^{n+1}.\label{eq:DtZetaj}
\end{equation}
Now, using (\ref{eq:zetai_espl}), the definitions of $a_{i,j}$ and
$b_{i,j}$, and the fact that $0\leq\chi\leq1$, we obtain, for every
$s\in\mathbb{R}$,\begin{align*}
\vert\tilde{\zeta}'_{i,j}(s)\vert\,&=\,\frac{1}{a_{i,j}-b_{i,j}}\left|\chi'\left(\frac{s+a_{i,j}}{a_{i,j}-b_{i,j}}\right)\,\chi\left(\frac{a_{i,j}-s}{a_{i,j}-b_{i,j}}\right)-\chi\left(\frac{s+a_{i,j}}{a_{i,j}-b_{i,j}}\right)\,\chi'\left(\frac{a_{i,j}-s}{a_{i,j}-b_{i,j}}\right)\right|\nonumber\\
&\leq\,\frac{2\,\Vert\chi'\Vert_{\infty}}{(\tilde{\rho}_{j})^{\frac{1}{p_{i}}}-\rho_{j+1}^{\frac{1}{p_{i}}}}\,.
\end{align*} Then, for every $(x,t)\in\mathbb{R}^{n+1}$, we get 
\begin{equation}
\vert\partial_{x_{i}}\tilde{\zeta}_{j}^{\frac{1}{p_{i}}}(x,t)\vert^{p_{i}}\,=\,\psi_{j}(t)\,\vert\tilde{\zeta}'_{i,j}(x_{i})\vert^{p_{i}}\,\prod_{\ell\neq i}\tilde{\zeta}_{\ell,j}^{p_{\ell}}(x_{\ell})\,\le\left[\frac{2\,\Vert\chi'\Vert_{\infty}}{(\tilde{\rho}_{j})^{\frac{1}{p_{i}}}-\rho_{j+1}^{\frac{1}{p_{i}}}}\right]^{p_{i}}.\label{eq:DiZeta}
\end{equation}
At this point, we consider the function $g_{i}:[0,\infty)\rightarrow[0,\infty)$
defined by 
\[
g_{i}(s):=\,s^{\frac{1}{p_{i}}},\,\,\,\,\,\,\,\,\,\,\,\,i\in\{1,\ldots,n\}.
\]
By Lagrange's mean value theorem, there exists some $\kappa_{i,j}\in(\rho_{j+1},\tilde{\rho}_{j})$
such that 
\begin{equation}
(\tilde{\rho}_{j})^{\frac{1}{p_{i}}}-\rho_{j+1}^{\frac{1}{p_{i}}}\,=\,g{}_{i}(\tilde{\rho}_{j})-g_{i}(\rho_{j+1})\,=\,\frac{\tilde{\rho}_{j}-\rho_{j+1}}{p_{i}}\,\kappa_{i,j}^{\frac{1}{p_{i}}\,-\,1},\label{eq:Lagrange}
\end{equation}
Joining (\ref{eq:DiZeta}) and (\ref{eq:Lagrange}), using the definitions
of $\rho_{j+1}$ and $\tilde{\rho}_{j}$, together with the inequalities
$\kappa_{i,j}<\tilde{\rho}_{j}<\rho$ and $p_{i}>1$, and recalling
that $\Vert\chi'\Vert_{\infty}\geq1$, $\sigma\in(0,1)$ and $\mathcal{P}:=\max\,\{p_{1},\ldots,p_{n}\}$,
we obtain 
\begin{equation}
\vert\partial_{x_{i}}\tilde{\zeta}_{j}^{\frac{1}{p_{i}}}\vert^{p_{i}}\,\leq\,\frac{(2\,p_{i}\,\Vert\chi'\Vert_{\infty})^{p_{i}}\,\kappa_{i,j}^{p_{i}-1}}{(\tilde{\rho}_{j}-\rho_{j+1})^{p_{i}}}\,\le\,\frac{(2^{j+3}\,\mathcal{P}\,\Vert\chi'\Vert_{\infty})^{\mathcal{P}}\,\rho^{p_{i}-1}}{(1-\sigma)^{p_{i}}\,\rho^{p_{i}}}\,\leq\,\frac{(2^{j+3}\,\mathcal{P}\,\Vert\chi'\Vert_{\infty})^{\mathcal{P}}}{(1-\sigma)^{\mathcal{P}}\rho}\,.\label{eq:DiZetaj2}
\end{equation}
Finally, taking $c=(8\,\mathcal{P}\,\Vert\chi'\Vert_{\infty})^{\mathcal{P}}$,
from (\ref{eq:DtZetaj}) and (\ref{eq:DiZetaj2}) we get
\[
0\,\leq\,\partial_{t}\tilde{\zeta}_{j}\,\leq\,\frac{c\,2^{j}}{(1-\sigma)\,\rho}\,\,\,\,\,\,\,\,\,\,\,\,\,\,\,\mathrm{and}\,\,\,\,\,\,\,\,\,\,\,\,\,\,\,\vert\partial_{x_{i}}\tilde{\zeta}_{j}^{\frac{1}{p_{i}}}\vert^{p_{i}}\,\leq\,\frac{c\,2^{j\mathcal{P}}}{(1-\sigma)^{\mathcal{P}}\rho}\,\,\,\,\,\,\,\,\mathrm{for\,\,every}\,\,i\in\{1,\ldots,n\}.
\]

\begin{singlespace}
\noindent \medskip{}

\noindent \textbf{Acknowledgments. }We would like to thank the reviewer
for his/her valuable comments, which helped to improve our paper.
This work has been partially supported by the Wallenberg AI, Autonomous
Systems and Software Program (WASP) funded by the Knut and Alice Wallenberg
Foundation. The authors are members of the Gruppo Nazionale per l'Analisi
Matematica, la Probabilità e le loro Applicazioni (GNAMPA) of the
Istituto Nazionale di Alta Matematica (INdAM). P$.\,$Ambrosio and
G$.\,$Cupini have been partially supported through the INdAM--GNAMPA
2026 Project ``Esistenza e regolarità per soluzioni di equazioni
ellittiche e paraboliche anisotrope'' (CUP E53C25002010001). S$.\,$Ciani
has been partially supported through the INdAM--GNAMPA 2026 Project
``Classificazione delle soluzioni di equazioni evolutive non locali''
(CUP E53C25002010001).\bigskip{}

\noindent \textbf{Declarations.} On behalf of all authors, the corresponding
author states that there is no conflict of interest.\bigskip{}

\noindent \textbf{Data availability.} This manuscript has no associated
data.\addcontentsline{toc}{section}{References}
\end{singlespace}

\begin{singlespace}

\lyxaddress{\noindent \textbf{$\quad$}\\
$\hspace*{1em}$\textbf{Pasquale Ambrosio} \\
Department of Mathematics, Uppsala University\\
P.O. Box 524, 751 20 Uppsala, Sweden.\\
\textit{E-mail address}: pasquale.ambrosio@math.uu.se}

\lyxaddress{\noindent $\hspace*{1em}$\textbf{Simone Ciani}\\
Dipartimento di Matematica, Università di Bologna\\
Piazza di Porta S. Donato 5, 40126 Bologna, Italy.\\
\textit{E-mail address}: simone.ciani3@unibo.it}

\lyxaddress{\noindent $\hspace*{1em}$\textbf{Giovanni Cupini}\\
Dipartimento di Matematica, Università di Bologna\\
Piazza di Porta S. Donato 5, 40126 Bologna, Italy.\\
\textit{E-mail address}: giovanni.cupini@unibo.it }
\end{singlespace}

\end{document}